%% file: excisev3.tex
\documentclass[a4paper,10pt]{amsart}

\input{preamble}
\input{macros}

\title{Cdh descent, cdarc descent, and Milnor excision}

\author[E. Elmanto]{Elden Elmanto}
\address{Department of Mathematics\\
Harvard University\\
1 Oxford St.\\
Cambridge, MA 02138\\
USA}
\email{\href{mailto:elmanto@math.harvard.edu}{elmanto@math.harvard.edu}}
\urladdr{\url{https://www.eldenelmanto.com/}}

\author[M. Hoyois]{Marc Hoyois}
\address{Fakultät für Mathematik\\
Universität Regensburg\\
Universitätsstr. 31\\
93040 Regensburg\\
Germany}
\email{\href{mailto:marc.hoyois@ur.de}{marc.hoyois@ur.de}}
\urladdr{\url{http://www.mathematik.ur.de/hoyois/}}

\author[R. Iwasa]{Ryomei Iwasa}
\address{K\o benhavns Universitet\\
Institut for Matematiske Fag\\
Universitetsparken 5\\
2100 K\o benhavn\\
Denmark}
\email{\href{mailto:ryomei@math.ku.dk}{ryomei@math.ku.dk}}
\urladdr{http://ryomei.com/}

\author[S. Kelly]{Shane Kelly}
\address{Department of Mathematics\\
Tokyo Institute of Technology\\
2-12-1 Ookayama, Meguro-ku\\
Tokyo 152-8551, Japan}
\email{\href{mailto:shanekelly@math.titech.ac.jp}{shanekelly@math.titech.ac.jp}}
\urladdr{http://www.math.titech.ac.jp/~shanekelly/}

\date{\today}

\begin{document}

\begin{abstract} 
	We give necessary and sufficient conditions for a cdh sheaf to satisfy Milnor excision, following ideas of Bhatt and Mathew. Along the way, we show that the cdh $\infty$-topos of a quasi-compact quasi-separated scheme of finite valuative dimension is hypercomplete, extending a theorem of Voevodsky to nonnoetherian schemes. As an application, we show that if $E$ is a motivic spectrum over a field $k$ which is $n$-torsion for some $n$ invertible in $k$, then the cohomology theory on $k$-schemes defined by $E$ satisfies Milnor excision.
\end{abstract}

\maketitle

\parskip 0pt
\tableofcontents

\parskip 0.2cm

\section{Introduction}

A \emph{Milnor square} is a cartesian square of rings of the form
\begin{equation*} \label{eq:milnor-intro}
\begin{tikzcd}
A \ar{r} \ar{d} & B \ar{d} \\
A/I \ar{r}  & B/J\rlap.
\end{tikzcd}
\end{equation*}
This class of squares was introduced by Milnor in \cite[\sectsign 2]{milnor-ktheory}. In modern language, he proved that such a square induces a cartesian square of categories
\[
\begin{tikzcd}
\Proj(A) \ar{r} \ar{d} & \Proj(B) \ar{d} \\
\Proj(A/I)\ar{r} & \Proj(B/J)\rlap,
\end{tikzcd}
\]
where $\Proj(A)$ denotes the category of finitely generated projective left $A$-modules.
We will say that a functor of rings satisfies \emph{Milnor excision} if it sends Milnor squares to cartesian squares. Thus, $\Proj(-)$ satisfies Milnor excision. A closely related invariant that satisfies Milnor excision is Weibel's homotopy $K$-theory spectrum $\KH(-)$ \cite[Theorem 2.1]{Weibel}. Algebraic $K$-theory itself does not satisfy Milnor excision, though it satisfies a weaker ``pro-excision'' property \cite{Morrow}; the failure of Milnor excision for algebraic $K$-theory is discussed in more details in \cite{land-tamme}.

Milnor excision for invariants of \emph{commutative} rings was further explored by Bhatt and Mathew in \cite{arc}.
They introduce a refinement of Voevodsky's h topology on the category of schemes, called the \emph{arc topology}, with the property that every arc sheaf satisfies Milnor excision. Moreover, they show that the difference between h descent and arc descent is a certain excision property involving absolutely integrally closed valuation rings, which is easy to check in some cases \cite[Theorem 1.7]{arc}. Examples of arc sheaves include étale cohomology with torsion coefficients and perfect complexes on the perfection of $\bF_p$-schemes.

It turns out that Milnor squares are in fact covers for a ``completely decomposed'' version of the arc topology.
A qcqs morphism of schemes $Y\to X$ is called a \emph{cdarc cover} if the induced map $\Maps(\Spec V,Y)\to \Maps(\Spec V, X)$ is surjective for every henselian valuation ring $V$ of rank $\leq 1$. Together with open covers, cdarc covers define a topology on the category of schemes called the \emph{cdarc topology}. The cdarc topology is to the cdh and Nisnevich topologies as the arc topology is to the h and fppf topologies. 
Our main result is the following theorem, which is a completely decomposed analogue of \cite[Theorem 4.1]{arc}:

\begin{thmA}[see Theorem~\ref{thm:main}] \label{thmA}
	Let $S$ be a scheme, $\sC$ a compactly generated $\infty$-category, and $\sF\colon \Sch_S^{\op}\to \sC$ a finitary cdh sheaf (see \sectsign\ref{sub:cdp}). Consider the following assertions:
	\begin{enumerate}
		\item $\sF$ satisfies henselian v-excision, i.e., for every henselian valuation ring $V$ over $S$ and every prime ideal $\p\subset V$, the following square is cartesian:
		\[
		\begin{tikzcd}
			\sF(V) \ar{r} \ar{d} & \sF(V_\p) \ar{d} \\
			\sF(V/\p) \ar{r} & \sF(\kappa(\p))\rlap;
		\end{tikzcd}
		\]
		\item $\sF$ satisfies Milnor excision;
		\item $\sF$ satisfies cdarc descent.
	\end{enumerate}
	In general, $(3)\Rightarrow(2)\Rightarrow(1)$. If $S$ has finite valuative dimension (see \sectsign\ref{sub:dim_v}), then $(1)\Rightarrow (2)$. If every compact object of $\sC$ is cotruncated, then $(1)\Rightarrow (3)$.
\end{thmA}

The v-excision property featured in this theorem was first considered by Huber and the fourth author in \cite{huber-kelly}, where it is shown that the presheaf $\Omega^n(-)$ of differential $n$-forms satisfies v-excision for equicharacteristic valuation rings \cite[Lemma 3.14]{huber-kelly}. Theorem~A applies to this example and shows that $\Omega^n_\cdh$ satisfies cdarc descent (see Corollary~\ref{cor:Omega}).

The assumption that every compact object of $\sC$ is cotruncated holds for example if $\sC=D(\Lambda)_{\leq 0}$ for some ring $\Lambda$ or if $\sC$ is the $\infty$-category of coconnective spectra.
Our theorem improves upon \cite[Theorem 4.1]{arc} in that this cotruncatedness assumption is not needed to deduce Milnor excision from henselian v-excision (the assumption that $S$ has finite valuative dimension is practically vacuous since $S$ is typically the spectrum of a field or Dedekind domain).
A key ingredient in the proof is the hypercompleteness of the cdh topology on nonnoetherian schemes, which is of independent interest:

\begin{thmB}[see Corollary~\ref{cor:cdh-hypercomplete}]
	Let $X$ be a qcqs scheme of finite valuative dimension. Then every sheaf in $\Shv_{\cdh}(\Sch_X^\fp)$ is the limit of its Postnikov tower. In particular, the $\infty$-topos $\Shv_{\cdh}(\Sch_X^\fp)$ is hypercomplete.
\end{thmB}

An example of an invariant satisfying arc descent is étale cohomology \[C^*_\et(-,A)\colon \Sch^\op \to D(\Z)\] with coefficients in a torsion abelian group $A$. For this invariant, henselian v-excision is trivial. Our criterion allows us to prove an analogous result for motivic cohomology of schemes over a field whose characteristic does not divide the torsion in $A$:

\begin{thmC}[see Corollary~\ref{cor:motivic-cohomology}]
	Let $k$ be a field and $A$ a torsion abelian group such that the exponential characteristic of $k$ acts invertibly on $A$. For any $q \in\Z$, the presheaf 
	\[
	C^*_\mot(-,A(q))\colon \Sch_k^{\op} \rightarrow \D(\Z)
	\] 
	satisfies cdarc descent. In particular, it satisfies Milnor excision.
\end{thmC}

More generally, we prove Milnor excision for cohomology theories defined by motivic spectra with suitable torsion. For a motivic spectrum $E\in\SH(S)$ and an $S$-scheme $X$, we denote by $E(X)$ the mapping spectrum from $\1_X$ to $E_X$ in $\SH(X)$.
It is well known that $E(-)\colon \Sch_S^\op\to \Spt$ is a finitary cdh sheaf of spectra.

\begin{thmD}[see Theorem~\ref{thm:main-excision}]
	Let $k$ be a field and $E \in \SH(k)$. Suppose that $E$ is $\phi$-torsion for some $\phi\in\GW(k)$ such that $\rk(\phi)$ is invertible in $k$. Then the presheaf of spectra $E(-)\colon\Sch_k^{\op} \rightarrow \Spt$ satisfies Milnor excision.
\end{thmD}

One of the impetus for trying to prove Milnor excision for a motivic spectrum $E$ over a perfect field is that it implies condition (G2) from \cite{shane-better} for the homotopy presheaves of $E$ \cite[Lemma 3.5(ii)]{kelly-morrow}.\footnote{The condition (G1) (resp.\ (G2)) asks that the horizontal (resp.\ vertical) morphisms in the square of Theorem~A be injective (resp.\ surjective).}
In forthcoming work \cite{sh-excision}, we will show that the torsion assumptions in Theorems~C and~D are in fact not necessary, and hence prove condition (G2) for any motivic spectrum.

\changelocaltocdepth{1}
\subsection*{Notation}

We denote by $\Spc$ the $\infty$-category of spaces and by $\Spt$ that of spectra.

For an $\infty$-category $\sC$, $\PSh(\sC)$ is the $\infty$-category of presheaves of spaces on $\sC$ and $\PSh_\Sigma(\sC)\subset \PSh(\sC)$ is the full subcategory of presheaves that transform finite coproducts into finite products (also called $\Sigma$-presheaves). If $\tau$ is a topology on $\sC$, $\Shv_\tau(\sC)\subset\PSh(\sC)$ is the full subcategory of $\tau$-sheaves.

If $X$ is a scheme, $\Sch_X$ is the category of $X$-schemes, $\Sch_X^\qcqs$ is the category of quasi-compact quasi-separated (qcqs) $X$-schemes, $\Sch_X^\mathrm{lfp}$ is the category of $X$-schemes locally of finite presentation, and $\Sch_X^\fp$ is the category of $X$-schemes of finite presentation.

\subsection*{Acknowledgments}

We would like to thank Akhil Mathew and Bhargav Bhatt for some useful discussions about the results of \cite{arc} and Benjamin Antieau for communicating Theorem~\ref{thm:temkin}.

This work was partially supported by the National Science Foundation under grant DMS-1440140, while the first two authors were in residence at the Mathematical Sciences Research Institute in Berkeley, California, during the ``Derived Algebraic Geometry'' program in spring 2019.

\changelocaltocdepth{2}
\section{Complements on the cdh topology}

In this section, we prove that the homotopy dimension of the cdh $\infty$-topos of a qcqs scheme $X$ is bounded by the valuative dimension of $X$, and we deduce that this $\infty$-topos is hypercomplete when the valuative dimension of $X$ is finite. 
For $X$ noetherian, the valuative dimension coincides with the Krull dimension and these results were previously proved by Voevodsky (\cite[Proposition 2.11]{unstable-voe} and \cite[Theorem 2.27, Proposition 3.8(3)]{cd-voe}).

We start by reviewing the definition of the cdh topology in \sectsign\ref{sub:cdp}.
In \sectsign\ref{sub:RZ}, we recall the definition of Riemann–Zariski spaces and establish some basic facts about them.
In \sectsign\ref{sub:dim_v}, we define the valuative dimension of a scheme, following Jaffard \cite{Jaffard}; the valuative dimension of an integral scheme turns out to be the Krull dimension of its Riemann–Zariski space.
Finally, in \sectsign\ref{sub:hodim}, we obtain the desired bound on the homotopy dimension of the cdh $\infty$-topos.

\subsection{The cdp, rh, and cdh topologies}
\label{sub:cdp}

Recall that a family of morphisms of schemes $\{Y_i\to X\}$ is \emph{completely decomposed} if, for every $x\in X$, there exists $i$ and $y\in Y_i$ such that $\kappa(x)\to\kappa(y)$ is an isomorphism.

The \emph{cdp topology} on the category of qcqs schemes is 
defined as follows: a sieve on $X$ is a cdp covering sieve if and only if it contains a completely decomposed family $\{Y_i\to X\}$ where each $Y_i\to X$ is proper and finitely presented. 
The \emph{rh topology} (resp.\ the \emph{cdh topology}) is the topology generated by the cdp topology and the Zariski topology (resp.\ the Nisnevich topology). The rh and cdh topologies are defined on all schemes.

\begin{rem}
	The cdh topology was introduced by Suslin and Voevodsky in \cite{sv-cycleschow}, and the rh topology by Goodwillie and Lichtenbaum in \cite{goodwillie-lichtenbaum}. Our definitions differ from theirs for nonnoetherian schemes, as we require coverings to be finitely presented.
\end{rem}

If $X$ is a scheme and $\tau\in\{\rh,\cdh\}$, the $\infty$-category $\Shv_\tau(\Sch_X^\mathrm{lfp})$ is an $\infty$-topos, which is equivalent to $\Shv_\tau(\Sch_X^\fp)$ when $X$ is qcqs (see \cite[Proposition C.5]{HoyoisGLV}). 
By \cite[Theorem 2.3]{GabberKelly} and \cite[Lemma 6.4.5.6]{HTT}, the points of this $\infty$-topos can be identified with the local schemes over $X$ having no nontrivial covering sieves. 
For the rh topology, these are precisely the valuation rings \cite[Proposition 2.1]{goodwillie-lichtenbaum}, and for the cdh topology these are the henselian valuation rings \cite[Theorem 2.6]{GabberKelly}.

 A \emph{splitting sequence} for a morphism $f\colon Y\to X$ is a sequence of closed subschemes
\[
\varnothing = Z_n\subset Z_{n-1}\subset \dotsb\subset Z_1\subset Z_0=X
\]
such that $f$ admits a section over the subscheme $Z_{i}-Z_{i+1}$ for all $i$.

\begin{lem}\label{lem:splitting}
	Let $X$ be a qcqs scheme and $f\colon Y\to X$ a morphism locally of finite presentation. Then $f$ is completely decomposed if and only if it admits a splitting sequence by finitely presented closed subschemes.
\end{lem}

\begin{proof}
	It is clear that $f$ is completely decomposed if it admits a splitting sequence.
	Conversely, suppose that $f$ is completely decomposed.
	Let $\sE$ be the poset of closed subschemes $Z\subset X$ such that the induced morphism $Y\times_XZ\to Z$ does not admit a finitely presented splitting sequence. Since $f$ is locally of finite presentation, $\sE$ is closed under cofiltered intersections. By Zorn's lemma, it remains to show that $\sE$ does not have a minimal element. Any $Z \in \sE$ is nonempty and hence has a maximal point $z \in Z$. Since $f$ is completely decomposed, it splits over $z$. Since $f$ is locally of finite presentation, it splits over $U_\red$ for some quasi-compact open neighborhood $U$ of $z$ in $Z$. Writing $Z_\red$ as a limit of finitely presented subschemes, we find a finitely presented nilimmersion $Z_1\subset Z$ such that $f$ splits over $U\times_ZZ_1$. Let $Z_2\subset Z_1$ be a finitely presented closed complement of $U\times_ZZ_1$. Then $Z_2\in\sE$, which shows that $Z$ is not minimal.
\end{proof}

\begin{rem}
	A completely decomposed morphism of finite type need not have a splitting sequence. For example, let $X=\N\cup\{\infty\}$ be the one-point compactification of $\N$ regarded as a profinite set, and let $Y\subset X\times\{\pm 1\}$ be the closed subset consisting of the points $(n,(-1)^n)$ and $(\infty,\pm 1)$. Viewing profinite sets as zero-dimensional affine $k$-schemes for some field $k$, the map $Y\to X$ is finite and completely decomposed, but it does not admit a section over any open neighborhood of $\infty\in X$.
\end{rem}

In the following proposition, $\sqcup$ denotes the topology generated by finite coproduct decompositions.

\begin{prop}\label{prop:coherent}
	Let $X$ be a scheme.
	\begin{enumerate}
		\item If $X$ is qcqs and $\tau\in\{\mathord\sqcup,\Zar,\Nis,\cdp,\rh,\cdh\}$, the $\infty$-topos $\Shv_\tau(\Sch_X^\mathrm{fp})$ is coherent and locally coherent.
		\item For $\tau\in\{\Zar,\Nis,\rh,\cdh\}$, the $\infty$-topos $\Shv_\tau(\Sch_X^\mathrm{lfp})$ is locally coherent.
	\end{enumerate}
\end{prop}

\begin{proof}
	By Lemma~\ref{lem:splitting}, for every completely decomposed family $\{Y_i\to X\}_{i\in I}$ with $X$ qcqs and $Y_i\to X$ finitely presented, there exists a finite subset $J\subset I$ such that the family $\{Y_i\to X\}_{i\in J}$ is completely decomposed.
	It follows that the given topologies on $\Sch_X^\fp$ are finitary, so (1) follows from \cite[Proposition A.3.1.3]{SAG}. Assertion (2) follows easily from (1).
\end{proof}

An \emph{abstract blowup square} is a cartesian square
\[
\begin{tikzcd}
	E \ar{r} \ar{d} & Y \ar{d}{p} \\
	Z\ar{r}{i} & X
\end{tikzcd}
\]
where $i$ is a closed immersion, $p$ is proper, and $p$ is an isomorphism over $X-i(Z)$. It is called \emph{finitely presented} if moreover $i$ and $p$ are finitely presented.

Recall that a \emph{cd-structure} on a category $\mathcal C$ is a collection of commutative squares in $\mathcal C$ \cite{cd-voe}. The associated Grothendieck topology on $\mathcal C$ is the minimal topology containing the sieves generated by these squares, as well as the empty sieve over the initial object. For example, the Zariski and Nisnevich topologies on the category of qcqs schemes are associated with standard cd-structures (for a discussion of the Nisnevich case with no noetherian hypotheses, see \cite[Appendix A]{norms}).

\begin{prop}\label{prop:cdp}
	Let $X$ be a qcqs scheme.
	\begin{enumerate}
		\item The cdp topology on $\Sch_X^\fp$ is associated with the cd-structure consisting of finitely presented abstract blowup squares.
		\item A presheaf $\sF$ on $\Sch_X^\fp$ is a cdp sheaf if and only if $\sF(\varnothing)=*$ and $\sF$ sends finitely presented abstract blowup squares to cartesian squares.
	\end{enumerate}
\end{prop}

\begin{proof}
	The second statement follows from the first and \cite[Theorem 3.2.5]{AHW}. 
	Let $f\colon Y\to X$ be completely decomposed, proper, and finitely presented. We must show that $f$ is covering for the topology $\tau$ associated with the cd-structure given by abstract blowup squares. By Lemma~\ref{lem:splitting}, $f$ admits a finitely presented splitting sequence $\varnothing=Z_n\subset\dotsb\subset Z_0=X$, and we prove the claim by induction on $n$. If $n=0$, then $X$ is empty and the claim is trivial. Suppose $n\geq 1$. Then $f$ admits a section over $X-Z_1$. Writing the schematic image of this section as a cofiltered limit of finitely presented closed subschemes of $Y$, we find such a closed subscheme $Y_1\subset Y$ such that $Y_1\to X$ is an isomorphism over $X-Z_1$ \cite[Tag 0CNG]{stacks}. Thus, $Y_1$ and $Z_1$ form an abstract blowup square over $X$. By construction, the base change of $f$ to $Y_1$ has a section and in particular is a $\tau$-cover.
	The base change of $f$ to $Z_1$ has a finitely presented splitting sequence of length $n-1$, and hence it is a $\tau$-cover by the induction hypothesis. It follows that $f$ is a $\tau$-cover, as desired.
\end{proof}

We shall say that a presheaf on some category $\sC$ of schemes is \emph{finitary} if it sends cofiltered limits of \emph{qcqs} schemes\footnote{By a \emph{cofiltered limit} of schemes, we will always mean the limit of a cofiltered diagram of schemes with affine transition morphisms.} to colimits. We denote by $\PSh^\fin(\sC)\subset \PSh(\sC)$ the full subcategory of finitary presheaves.

If $X$ is qcqs, recall that $\Sch_X^\qcqs$ can be identified with the full subcategory of $\Pro(\Sch_X^\fp)$ spanned by limits of cofiltered diagrams with affine transition morphisms \cite[Tags 01ZC and 09MV]{stacks}. This implies that there is an equivalence of $\infty$-categories
\[
\PSh(\Sch_X^\fp) \simeq \PSh^\fin(\Sch_X^\qcqs).
\]

\begin{prop}\label{prop:finitary-sheaves}
	Let $S$ be a scheme.
	\begin{enumerate}
		\item If $S$ is qcqs and $\tau\in\{\mathord\sqcup,\Zar,\Nis,\cdp,\rh,\cdh\}$, there is an equivalence of $\infty$-categories
		\[
		\Shv_\tau(\Sch_S^\fp) \simeq\Shv_\tau^\fin(\Sch_S^\qcqs).
		\]
		\item For $\tau\in\{\Zar,\Nis,\rh,\cdh\}$, there is an equivalence of $\infty$-categories
		\[
		\Shv_\tau(\Sch_S^\mathrm{lfp}) \simeq \Shv_\tau^\fin(\Sch_S).
		\]
	\end{enumerate}
\end{prop}

\begin{proof}
	(1) By Proposition~\ref{prop:cdp}, each of these topologies $\tau$ is associated with an obvious cd-structure, whose squares will be called $\tau$-squares. By \cite[Theorem 3.2.5]{AHW}, a presheaf on $\Sch_S^\qcqs$ or $\Sch_S^\fp$ is a $\tau$-sheaf if and only if it sends the empty scheme to a contractible space and every $\tau$-square to a cartesian square.
	If $X$ is a cofiltered limit of finitely presented $S$-schemes $X_\alpha$, then every $\tau$-square over $X$ is the pullback of a $\tau$-square over $X_\alpha$ for some $\alpha$ \cite[Tags 07RP and 081F]{stacks}. It follows that a finitary presheaf on $\Sch_S^\qcqs$ is a $\tau$-sheaf if and only if its restriction to $\Sch_S^\fp$ is. 
	
	(2) Note that both sides are Zariski sheaves in $S$, so we may assume that $S$ is qcqs. In this case it is clear that the restriction functors to the $\infty$-categories from (1) are equivalences.
\end{proof}

\begin{rem}
	The conclusion of Proposition~\ref{prop:finitary-sheaves} does not hold for $\tau\in\{\et,\mathrm{fppf},\mathrm{eh},\h\}$. This is one of the main technical advantages of completely decomposed topologies. Another one is the hypercompleteness property that we will prove in \sectsign\ref{sub:hodim}.
\end{rem}

\begin{rem}\label{rem:abu}
	If $X$ is a qcqs scheme, every abstract blowup square $(Y,E)\to (X,Z)$ is a cofiltered limit of finitely presented abstract blowup squares $(Y_\alpha,E_\alpha)\to (X,Z_\alpha)$ \cite[Tags 09ZP and 09ZR]{stacks}. As a result, a finitary cdp sheaf on $\Sch_X^\qcqs$ takes every abstract blowup square to a cartesian square. However, a finitary cdp sheaf need not be a sheaf for the topology generated by completely decomposed proper coverings, since Lemma~\ref{lem:splitting} does not apply to morphisms of finite type.
\end{rem}

\subsection{Riemann–Zariski spaces}
\label{sub:RZ}

By a \emph{generic point} of a scheme $X$ we mean a generic point of an irreducible component of $X$. We denote by $X^\mathrm{gen}\subset X$ the set of generic points with the induced topology. A morphism of schemes $f\colon Y\to X$ is \emph{narrow} if $f(Y)\cap X^\mathrm{gen}=\varnothing$.

We shall say that a scheme is \emph{quasi-integral} if it is reduced and locally has finitely many generic points. If $X$ is quasi-integral, then $X^\mathrm{gen}$ is a discrete space.

\begin{defn}
	Let $X$ be a quasi-integral scheme.
	A \emph{modification} of $X$ is a proper morphism $f\colon Y\to X$ such that $Y$ is reduced, $f$ induces a bijection $Y^\mathrm{gen}\simeq X^\mathrm{gen}$, and for every $\eta\in Y^\mathrm{gen}$, the residual field extension $\kappa(f(\eta))\to \kappa(\eta)$ is an isomorphism.
\end{defn}

If $f\colon Y\to X$ is a modification, there exists a dense open subscheme $U\subset X$ such that $f^{-1}(U)\to U$ is an isomorphism \cite[Tag 02NV]{stacks}. We denote by $\sM_X\subset \Sch_X$ the full subcategory of modifications of $X$.
Note that $\sM_X$ is a poset, since a morphism from a reduced scheme to a separated scheme is determined by its restriction to the generic points.

\begin{defn}\label{def:RZ}
	Let $X$ be a quasi-integral scheme.
	The \emph{Riemann–Zariski space} of $X$ is the limit of all modifications of $X$ in the category of ringed spaces:
	\[
	\RZ(X) = \lim_{Y\in\sM_X} Y.
	\]
\end{defn}

\begin{rem}
	In \cite[\sectsign 2.1]{temkin-rz}, Temkin defines the relative Riemann–Zariski space $\RZ_Y(X)$ for any quasi-compact separated morphism of schemes $Y\to X$ as the limit of all factorizations $Y \to X_i \to X$ into a schematically dominant morphism $Y \to X_i$ followed by a proper morphism $X_i \to X$. The ringed space $\RZ(X)$ of Definition~\ref{def:RZ} is exactly $\RZ_{X^\mathrm{gen}}(X)$.
\end{rem}

In the definition of $\RZ(X)$, it is possible to replace $\sM_X$ by several other categories without changing the limit. 
Let $\sM_X^\mathrm{bl}\subset\sM_X$ be the subcategory of blowups\footnote{By a \emph{blowup} we will always mean a blowup with finitely presented center, so that blowups are proper.} of $X$ with narrow centers.
Let $\sN_X$ be the category of proper $X$-schemes $Y$ such that the map $(Y_\eta)_\red\to \eta$ is an isomorphism for every $\eta\in X^\mathrm{gen}$, and let $\sN_X^\fp= \sN_X\cap\Sch_X^\fp$. Note that $\sM_X\subset \sN_X$.

\begin{lem}\label{lem:RZcategories}
	Let $X$ be a quasi-integral scheme.
	\begin{enumerate}
		\item The categories $\sN_X$, $\sN_X^\fp$, $\sM_X$, and $\sM_X^\mathrm{bl}$ are cofiltered.
		\item The inclusion $\sM_X\subset \sN_X$ is coinitial. If $X$ is qcqs, the inclusion $\sM_X^\mathrm{bl}\subset\sM_X$ is coinitial.
		\item Suppose $X$ qcqs and let $F\colon \sN_X\to \sC$ be a functor such that, for any cofiltered diagram $\{Y_i\}$ in $\sN_X^\fp$ whose transition maps are closed immersions, $F(\lim_i Y_i)\simeq \lim_i F(Y_i)$. Then $F$ is right Kan extended from $\sN_X^\fp$.
	\end{enumerate}
\end{lem}

\begin{proof}
	(1,2) The categories $\sN_X$ and $\sN_X^\fp$ are cofiltered because they are closed under finite limits in $\Sch_X$. The inclusion $\sM_X\subset \sN_X$ has a right adjoint that sends $Y$ to the closure of $X^\mathrm{gen}$ in $Y$ with reduced scheme structure (in other words, it discards the narrow components of $Y$). Hence, $\sM_X$ is also cofiltered and $\sM_X\subset \sN_X$ is coinitial.
	The poset $\sM_X^\mathrm{bl}$ is cofiltered by \cite[Tag 080A]{stacks}. If $X$ is qcqs, any modification of $X$ is an isomorphism over a dense quasi-compact open $U\subset X$, and by \cite[Corollaire 5.7.12]{GrusonRaynaud} it can be refined by a $U$-admissible blowup of $X$. This shows that the inclusion $\sM_X^\mathrm{bl}\subset \sM_X$ is coinitial.
	
	(3) Recall that the category of qcqs $X$-schemes is equivalent to the full subcategory of $\Pro(\Sch_X^\fp)$ spanned by limits of cofiltered diagrams  with affine transition maps \cite[Tags 01ZC and 09MV]{stacks}.
	By \cite[Tag 09ZR]{stacks}, any proper $X$-scheme $Y$ is the limit of a cofiltered diagram of finitely presented proper $X$-schemes $Y_i$ whose transition maps are closed immersions.
	Moreover, for every $\eta\in X^\mathrm{gen}$, the scheme $(Y_i)_\eta$ is noetherian, so the closed immersion $Y_\eta\to (Y_i)_\eta$ is eventually an isomorphism. In particular, if $Y\in\sN_X$, then $Y_i$ is eventually in $\sN_X^\fp$. This shows that $\sN_X$ is equivalent to the full subcategory of $\Pro(\sN_X^\fp)$ spanned by limits of cofiltered diagrams of closed immersions, which implies the claim.
\end{proof}

\begin{cor}\label{cor:RZcategories}
	Let $X$ be a quasi-integral scheme. Then $\RZ(X)\simeq \lim_{Y\in \sN_X}Y$.
	If $X$ is qcqs, then $\RZ(X)\simeq \lim_{Y\in\sC}Y$ for $\sC$ either $\sM_X^\mathrm{bl}$ or $\sN_X^\fp$.
\end{cor}

\begin{rem}
	The analogue of Lemma~\ref{lem:RZcategories}(3) is not true for the category $\sM_X$, because every cofiltered diagram of closed immersions in $\sM_X$ is constant. 
	It is thus necessary to pass to the larger category $\sN_X$ to express $\RZ(X)$ as a limit of finitely presented $X$-schemes.
\end{rem}

\begin{rem}
	If $X'\to X$ is a morphism of quasi-integral schemes that sends generic points to generic points and $Y\in\sN_X$, then $Y\times_XX'\in\sN_{X'}$. It thus follows from Corollary~\ref{cor:RZcategories} that $X\mapsto\RZ(X)$ is functorial with respect to such morphisms.
\end{rem}

\begin{prop}\label{prop:RZ-properties}
	Let $X$ be a quasi-integral scheme.
	\begin{enumerate}
		\item $\RZ(X)$ is a locally ringed space whose underlying topological space is locally spectral; it is spectral if $X$ is qcqs.
		\item If $Y$ is a modification of $X$, there is a canonical isomorphism $\RZ(Y)\simeq \RZ(X)$.
		\item If $\{X_i\}$ are the irreducible components of $X$, then $\RZ(X)\simeq \bigcoprod_i \RZ(X_i)$.
		\item If $U\subset X$ is open, the canonical map $\RZ(U)\to \RZ(X)\times_XU$ is an isomorphism.
	\end{enumerate}
\end{prop}

\begin{proof}
	(1) This follows from Lemma~\ref{lem:RZcategories}(1) and the fact that the forgetful functors from locally ringed spaces to ringed spaces and from (locally) spectral spaces to topological spaces preserve cofiltered limits (the former is easy and the latter is \cite[Tag 0A2Z]{stacks}).
	
	(2) We have $\sM_Y=(\sM_X)_{/Y}$ and the forgetful functor $(\sM_X)_{/Y}\to\sM_X$ is coinitial since $\sM_X$ is cofiltered, by Lemma~\ref{lem:RZcategories}(1).
	
	(3) This follows from (2) since $\bigcoprod_i X_i\to X$ is a modification of $X$.
	
	(4) It suffices to prove the result when $U$ is affine. In that case, we will show that the pullback functor $\sM_X\to\sM_U$ is coinitial.
	Since $\sM_X$ and $\sM_U$ are cofiltered posets by Lemma~\ref{lem:RZcategories}(1), it will suffice to show that every modification of $U$ is refined by the pullback of a modification of $X$.
	 By Lemma~\ref{lem:RZcategories}(2), any modification of $U$ is refined by a projective modification $Y\to U$. Since $U$ is affine, we can embed $Y$ in a projective space $\P^n_U$. Let $\bar Y$ be the closure of $Y$ in $\P^n_X$ and let $X'\subset X$ be the union of the irreducible components disjoint from $U$. Then $\bar Y\sqcup X'$ is a modification of $X$ lifting $Y$.
\end{proof}

Our next goal is to generalize Proposition~\ref{prop:RZ-properties}(4) to étale morphisms. Note that if $X$ is quasi-integral and $V\to X$ is an étale morphism, then $V$ is also quasi-integral.\footnote{However, if $X$ is integral, then $V$ need not be locally integral (unless $X$ is geometrically unibranch). This is the reason for considering quasi-integral schemes.}

\begin{lem}\label{lem:etaleblowup}
	Let $X$ be a qcqs scheme with finitely many irreducible components and $V\to X$ a quasi-compact separated étale morphism. For every blowup $V'\to V$ with narrow center, there exists a blowup $X'\to X$ with narrow center such that $X'\times_XV \to V$ factors through $V'$.
\end{lem}

\begin{proof}
	Since $X$ has finitely many generic points, there exists a dense quasi-compact open $U\subset X$ such that $V\times_XU\to U$ is finite \cite[Tag 02NW]{stacks}.
	By Zariski's main theorem, we can find a factorization $V\to \bar V\to X$ where $V\to\bar V$ is a dense open immersion, $p\colon \bar V\to X$ is finite and finitely presented, and $p$ is étale over $U$ \cite[Tags 0F2N and 0AXP]{stacks}. By Gruson–Raynaud flatification \cite[Tag 0815]{stacks}, there exists a $U$-admissible blowup of $X$ such that the strict transform of $p$ is flat and finitely presented, hence finite locally free. Since a composition of blowups is a blowup \cite[Tag 080B]{stacks}, we are reduced to the case where $p$ is finite locally free. Let $\sI\subset\sO_{\bar V}$ be a finitely generated ideal whose restriction to $V$ cuts out the center of the given blowup $V'\to V$, and let $\sJ\subset\sO_X$ be the $0$th Fitting ideal of $\sO_{\bar V}/\sI$. Then $\sJ$ is a finitely generated ideal cutting out a narrow subscheme of $X$; indeed, the zero locus of $\sJ$ is the image by $p$ of the zero locus of $\sI$, which is narrow since $V$ is dense in $\bar V$.
	 As in the proof of \cite[Tag 0812]{stacks}, there exists a finitely generated ideal $\sI'\subset \sO_V$ such that $\sJ_V = \sI_V\sI'$. Hence, by \cite[Tag 080A]{stacks}, the blowup of $\sJ$ in $X$ has the desired property.
\end{proof}

\begin{prop}\label{prop:coinitial}
	Let $X$ be a quasi-integral qcqs scheme and $V\to X$ a quasi-compact separated étale morphism. Then the pullback functor $\sM_X\to \sM_V$ is coinitial.
\end{prop}

\begin{proof}
	Since the categories $\sM_X$ and $\sM_V$ are cofiltered posets, it suffices to show that the poset $\sM_X\times_{\sM_V}(\sM_V)_{/V'}$ is nonempty for every $V'\in\sM_V$. This follows from Lemma~\ref{lem:RZcategories}(2) and Lemma~\ref{lem:etaleblowup}.
\end{proof}

\begin{cor}\label{cor:etale-cosheaf}
	Let $X$ be a quasi-integral scheme and $V\to X$ an étale morphism. Then the canonical map of locally ringed spaces $\RZ(V)\to \RZ(X)\times_XV$ is an isomorphism.
\end{cor}

\begin{proof}
	The assertion is local on $V$ and $X$ by Proposition~\ref{prop:RZ-properties}(4), so we can assume $X$ qcqs and $V\to X$ quasi-compact and separated. Then the result follows from Proposition~\ref{prop:coinitial}.
\end{proof}

\begin{rem}
Corollary~\ref{cor:etale-cosheaf} implies that $X\mapsto \RZ(X)$ is an étale cosheaf of locally ringed spaces (hence of topological spaces) on quasi-integral schemes.
\end{rem}

We conclude this subsection by recalling the valuative description of Riemann–Zariski spaces, following \cite[\sectsign 3.2]{temkin-curves}. 
If $X$ is an integral scheme with function field $K$, we denote by $\Val(X)$ the set of valuation rings $R\subset K$ centered on $X$, i.e., such that $R$ contains $\sO_{X,x}$ for some $x\in X$. Given a finite subset $S\subset K^\times$ and an open subset $U\subset X$, let $\Val(U,S)\subset \Val(X)$ be the subset of valuation rings centered on $U$ and containing $S$. The sets $\Val(U,S)$ form a basis for a topology on $\Val(X)$ (which is closed under finite intersections). In this topology, $R_0$ is a specialization of $R_1$ if and only if $R_0\subset R_1$.

Given $R\in\Val(X)$, the morphism $\Spec R\to X$ lifts uniquely to every modification of $X$ (by the valuative criterion of properness). Taking the images of the closed point of $\Spec R$ in every modification of $X$ defines a map
\[
\phi\colon \Val(X)\to\RZ(X).
\]

\begin{prop}\label{prop:RZ=Val}
	Let $X$ be an integral scheme. Then $\phi\colon \Val(X)\to\RZ(X)$ is a homeomorphism.
\end{prop}

\begin{proof}
	Let $K$ be the function field of $X$.
	For $x\in\RZ(X)$, let $\sO_x = \colim_{Y\in\sM_X}\sO_{Y,x_Y}\subset K$ be the stalk of the structure sheaf at $x$. We claim that $\sO_x$ is a valuation ring.
	For $S\subset K^\times$ a finite subset, let $X\langle S\rangle$ be the schematic image of the morphism $\Spec K\to (\P^1)^{\lvert S\rvert}\times X$ induced by $S$, and let $X[S]= X\langle S\rangle \cap (\A^{\lvert S\rvert}\times X)$. Then $X\langle S\rangle$ is a modification of $X$ and $S\subset \sO(X[S])$.
	Now if $f\in K^\times$, then $X[f]$ and $X[f^{-1}]$ form an open covering of $X\langle f\rangle$, so $f$ or $f^{-1}$ belongs to $\sO_x$.
	
	For any $x\in\RZ(X)$, we have $\phi(\sO_x)=x$ by the uniqueness part of the valuative criterion. On the other hand, if $R\in\Val(X)$, $\sO_{\phi(R)}$ is a valuation ring dominated by $R$, hence is equal to $R$. Thus, $\phi$ is bijective.

It remains to show that $\phi$ is open and continuous. Given a finite subset $S\subset K^\times$ and an open $U\subset X$, $\phi(\Val(U,S))$ is the preimage of $U[S]$ in $\RZ(U\langle S\rangle)=\RZ(U)\subset \RZ(X)$, hence it is open in $\RZ(X)$. A basis of the topology of $\RZ(X)$ is given by $\RZ(V)$ where $V$ is an affine open in some $Y\in\sM_X$ lying over some affine open $U\subset X$. If $U=\Spec A$ and $V=\Spec B$, then $B=A[S]$ for some finite subset $S\subset K^\times$ and $\phi^{-1}(\RZ(V))=\Val(U,S)$.
\end{proof}

\subsection{Valuative dimension}
\label{sub:dim_v}

The valuative dimension of a commutative ring was introduced by Jaffard in \cite[Chapitre IV]{Jaffard}. We start by recalling the definition.

If $A$ is an integral domain with fraction field $K$, its valuative dimension $\dim_v(A)$ is the supremum of $n$ over all chains
\[
A\subset R_0\subset R_1\subset \dotsb\subset R_n \subset K
\]
where each $R_i$ is a valuation ring and $R_i\neq R_{i+1}$. Equivalently, $\dim_v(A)$ is the supremum of the ranks of the valuation rings $R$ such that $A\subset R\subset K$.

If $A$ is any commutative ring, Jaffard defines
\[
\dim_v(A) = \sup_{\mathfrak p\in\Spec A} \dim_v(A/\mathfrak p).
\]
This recovers the previous definition when $A$ is integral, by \cite[p.\ 55, Lemme~1]{Jaffard}.

\begin{lem}\label{lem:dim_v}
	Let $f_1,\dotsc,f_n\in A$ generate the unit ideal. Then
	\[
	\dim_v(A) = \max_{1\leq i\leq n}\dim_v(A_{f_i}).
	\]
\end{lem}

\begin{proof}
	We can assume that $A$ is an integral domain. In this case it is obvious that $\dim_v(A_f)\leq \dim_v(A)$ for any $f\in A$. The other inequality comes from the following observation: any valuation ring $R$ containing $A$ must contain $A_{f_i}$ for some $i$. Indeed, there exists $i$ such that $f_i$ divides $f_j$ in $R$ for all $j$, hence divides $1$.
\end{proof}

If $X$ is a scheme, we define
\[
\dim_v(X) = \sup_{U\subset X} \dim_v(\sO(U))
\]
where $U$ ranges over the affine open subschemes of $X$. If $X=\Spec A$, it follows from Lemma~\ref{lem:dim_v} that $\dim_v(X)=\dim_v(A)$.

\begin{prop}\label{prop:dim_v}
	Let $X$ be a scheme.
	\begin{enumerate}
		\item If $\{U_i\}$ is an open cover of $X$, then $\dim_v(X)=\sup_i\dim_v(U_i)$.
		\item If $\{Z_i\}$ is a closed cover of $X$, then $\dim_v(X)=\sup_i\dim_v(Z_i)$.
		\item For any subscheme $Y\subset X$, $\dim_v(Y)\leq \dim_v(X)$.
		\item For any narrow subscheme $Y\subset X$, $\dim_v(Y)+1\leq\dim_v(X)$.
		\item If $Y\to X$ is a dominant integral morphism, then $\dim_v(Y)=\dim_v(X)$.
		\item If $X$ is quasi-integral and $Y\to X$ is a modification, then $\dim_v(Y)=\dim_v(X)$.
		\item For any $n\geq 0$, $\dim_v(\A^n\times X)=\dim_v(X)+n$.
		\item $\dim(X)\leq \dim_v(X)$.
		\item If $X$ is locally noetherian, $\dim_v(X)=\dim(X)$.
	\end{enumerate}
\end{prop}

\begin{proof}
	(1) This follows easily from the definitions and Lemma~\ref{lem:dim_v}.
	
	(2) By (1), we can assume $X=\Spec A$ and $Z_i=\Spec A/I_i$. Then $\{Z_i\}$ being a closed cover means that every prime ideal $\p \subset A$ contains $I_i$ for some $i$. Hence,
	\[
	\dim_v(X)=\sup_{\p} \dim_v(A/\p) =\sup_{i}\sup_{\p\supset I_i}\dim_v(A/\p)=\sup_i \dim_v(Z_i).
	\]
	
	(3) This is obvious for an open subscheme, so assume that $Y$ is closed. By (1), we can assume $X=\Spec A$ and $Y=\Spec A/I$. Then
	\[
	\dim_v(Y) = \sup_{\mathfrak p\supset I} \dim_v(A/\mathfrak p) \leq \sup_{\mathfrak p} \dim_v(A/\mathfrak p)=\dim_v(X).
	\]
	
	(4) There is a largest open subscheme $U\subset X$ in which $Y$ is closed. Every generic point of $U$ is a generic point of $X$, so $Y$ is narrow in $U$. Thus we may assume that $Y$ is closed in $X$, and by (1) we are reduced to the case $X=\Spec A$ and $Y=\Spec A/I$. Since $Y$ is narrow, every $\mathfrak p\supset I$ contains a strictly smaller prime ideal $\mathfrak p'$, hence
	\[
	\dim_v(Y)+1=\sup_{\mathfrak p\supset I} \dim_v(A/\mathfrak p)+1 \leq \sup_{\mathfrak p\supset I} \dim_v(A/\mathfrak p') \leq \dim_v(X),
	\]
	where the nontrivial inequality is \cite[p.\ 55, Remarque]{Jaffard}.
	
	(5) By (2), we can assume $Y\to X$ schematically dominant. Then the result follows from (1) and \cite[p.\ 58 Proposition 4]{Jaffard}.
	
	(6) By (1,2), we can assume that $X=\Spec A$ where $A$ is an integral domain with fraction field $K$. If $\Spec B\subset Y$ is an affine open, then $A\subset B\subset K$, so the inequality $\dim_v(Y)\leq \dim_v(X)$ follows from (1). Since $Y$ is proper, any valuation ring $R\subset K$ containing $A$ lifts to a morphism $\Spec R \to Y$ by the valuative criterion; hence $R$ contains $B$ for some affine open $\Spec B\subset Y$. It then follows from (3) that $\dim_v(Y) \geq \dim_v(X)$.
	
	(7–9) These statements follow from (1) and \cite[p.\ 60, Théorème 2]{Jaffard}, \cite[p.\ 56, Théorème 1]{Jaffard}, and \cite[p.\ 67, Corollaire 2]{Jaffard}, respectively.
\end{proof}

\begin{cor}\label{cor:dim_v}
	Let $X$ be a quasi-compact scheme of finite valuative dimension. Then any $X$-scheme of finite type has finite valuative dimension.
\end{cor}

\begin{proof}
	This follows from Proposition~\ref{prop:dim_v}(1,7,3).
\end{proof}

\begin{cor}
	Let $X$ be an integral scheme with function field $K$. Then $\dim_v(X)$ is the supremum of the ranks of the valuation rings of $K$ centered on $X$.
\end{cor}

\begin{proof}
	If $X$ is affine, this is the definition of $\dim_v(X)$. The general case follows from Proposition~\ref{prop:dim_v}(1), since any valuation ring centered on $X$ is centered on some affine open subscheme of $X$.
\end{proof}

\begin{prop}\label{prop:residue}
	Let $f\colon Y\to X$ be a morphism of schemes such that, for every generic point $\eta\in Y^\mathrm{gen}$, the transcendence degree of $\kappa(\eta)$ over $\kappa(f(\eta))$ is $\leq d$. Then $\dim_v(Y)\leq \dim_v(X)+d$.
\end{prop}

\begin{proof}
	By Proposition~\ref{prop:dim_v}(1,2) we can assume that $f$ is induced by an extension of integral domains $A\subset B$ with fraction fields $K\subset L$. Let $S\subset L$ be a valuation ring containing $B$, and let $R=S\cap K$. Then $R$ is a valuation ring containing $A$. The result now follows from the inequality $\rk(S) \leq \rk(R) + d$ \cite[VI \sectsign 10.3, Corollaire 2]{BourbakiCAlg}.
\end{proof}

\begin{rem}
	As a special case of Proposition~\ref{prop:residue}, if $Y\to X$ is a birational morphism of schemes in the sense of \cite[I.2.3.4]{EGA1new} (i.e., it induces a bijection $Y^\mathrm{gen}\simeq X^\mathrm{gen}$ and isomorphisms $\sO_{X,f(\eta)}\simeq \sO_{Y,\eta}$ for all $\eta\in Y^\mathrm{gen}$), then $\dim_v(Y)\leq \dim_v(X)$.
\end{rem}

\begin{lem}\label{lem:dimlim}
	Let $\{X_i\}$ be a cofiltered diagram in the category of topological spaces with limit $X$. Then $\dim(X) \leq \sup_i\dim(X_i)$.
\end{lem}

\begin{proof}
	Any strict chain of specializations in $X$ induces a strict chain in $X_i$ for some $i$.
\end{proof}

\begin{prop}
	Let $X$ be a quasi-integral scheme. Then $\dim_v(X) = \dim(\RZ(X))$.
\end{prop}

\begin{proof}
	By Propositions \ref{prop:RZ-properties}(3) and \ref{prop:dim_v}(2), we can assume $X$ integral.
	Let $K$ be the function field of $X$. If $R_0\subset R_1$ are valuation rings of $K$ centered on $X$, then $\phi(R_0)$ is a specialization of $\phi(R_1)$ in $\RZ(X)$. Since $\phi$ is injective by Proposition~\ref{prop:RZ=Val}, we deduce that $\dim_v(X)\leq \dim(\RZ(X))$. On the other hand, by Lemma~\ref{lem:dimlim} and Proposition~\ref{prop:dim_v}(6,8),
	\[
	\dim(\RZ(X)) \leq \sup_{Y\in\sM_X}\dim (Y) \leq\sup_{Y\in\sM_X} \dim_v(Y) =\dim_v(X).\qedhere
	\]
\end{proof}

The following lemma is a generalization of \cite[Lemma A.2]{shane-better}.

\begin{lem}\label{lem:finite-rank}
	Let $A$ be a commutative ring of finite valuative dimension. Then any morphism $A\to R$ where $R$ is a (henselian) valuation ring is a filtered colimit of morphisms $A\to R_\alpha$ where $R_\alpha$ is a (henselian) valuation ring of finite rank.
\end{lem}

\begin{proof}
	Write $R$ as the union of its finitely generated sub-$A$-algebras $A_\alpha$. Let $K_\alpha$ be the fraction field of $A_\alpha$ and let $R_\alpha = R\cap K_\alpha$. Since $A_\alpha$ has finite valuative dimension by Corollary~\ref{cor:dim_v}, $R_\alpha$ is a valuation ring of finite rank. If $R$ is henselian, we can replace each $R_\alpha$ by its henselization, which is a valuation ring of the same rank as $R_\alpha$ \cite[Tag 0ASK]{stacks}.
\end{proof}

\subsection{The homotopy dimension of the cdh $\infty$-topos}
\label{sub:hodim}

Recall that an $\infty$-topos $\sX$ has \emph{homotopy dimension $\leq d$} if every $d$-connective object of $\sX$ admits a global section, and it has \emph{finite homotopy dimension} if it has homotopy dimension $\leq d$ for some $d$ \cite[Definition 7.2.1.1]{HTT}. We say that $\sX$ is \emph{locally of finite homotopy dimension} if it is generated under colimits by objects $U$ such that $\sX_{/U}$ has finite homotopy dimension.

If $\sX$ has homotopy dimension $\leq d$, then it has cohomological dimension $\leq d$ in the sense that $H^n(\sX;\sA)=0$ for any $n>d$ and any sheaf of groups $\sA$ on $\sX$. Indeed, if $a\colon *\to K(\sA,n)$ represents a cohomology class in $H^n(\sX;\sA)$, then nullhomotopies of $a$ are equivalent to global sections of the pullback of $a$ along $0\colon *\to K(\sA,n)$, which is an $(n-1)$-connective object.

\begin{prop}[Lurie]
	\label{prop:hodim}
	Let $\sX$ be an $\infty$-topos.
	\begin{enumerate}
		\item If $\sX$ has homotopy dimension $\leq d$ and $\sF$ is a $n$-connective object of $\sX$, then $\Maps(*,\sF)$ is $(n-d)$-connective.
		\item If $\sX$ is locally of finite homotopy dimension, then every object of $\sX$ is the limit of its Postnikov tower. In particular, $\sX$ is hypercomplete.
	\end{enumerate}
\end{prop}

\begin{proof}
	\cite[Lemma 7.2.1.7]{HTT} and \cite[Proposition 4.1.5]{LurieTopoi}.
\end{proof}

	Let $\{X_i\}$ be a cofiltered diagram of qcqs algebraic spaces defining a pro-algebraic space $X$. We define the small Zariski and Nisnevich sites $X_\Zar$ and $X_\Nis$ to be the colimits of the small sites $(X_i)_\Zar$ and $(X_i)_\Nis$, so that
	\[
	\Shv(X_\Zar) \simeq \lim_i \Shv((X_i)_\Zar)\quad\text{and}\quad\Shv(X_\Nis) \simeq \lim_i \Shv((X_i)_\Nis).
	\]

\begin{lem}[Clausen–Mathew]\label{lem:CM}
	Let $X$ be a pro-algebraic space limit of a cofiltered diagram of qcqs algebraic spaces of Krull dimension $\leq d$. Then the $\infty$-topoi $\Shv(X_\Zar)$ and $\Shv(X_\Nis)$ have homotopy dimension $\leq d$.
\end{lem}

\begin{proof}
	\cite[Corollary 3.11, Theorem 3.12, and Theorem 3.17]{clausen-mathew}.
\end{proof}

If $X$ is a quasi-integral scheme, we can regard $\RZ(X)$ as a pro-scheme by taking the limit of Definition~\ref{def:RZ} in the category of pro-schemes.

\begin{prop}\label{prop:hodimRZ}
	Let $X$ be a quasi-integral qcqs scheme of finite valuative dimension $d$. Then the $\infty$-topoi $\Shv(\RZ(X)_\Zar)$ and $\Shv(\RZ(X)_\Nis)$ have homotopy dimension $\leq d$.
\end{prop}

\begin{proof}
	Every modification of $X$ has Krull dimension $\leq d$ by Proposition~\ref{prop:dim_v}(6,8). Hence, the result follows from Lemma~\ref{lem:CM}.
\end{proof}

Let $X$ be a quasi-integral qcqs scheme. For most of this subsection we will work with an auxiliary category of generically finite $X$-schemes: let $\GFin_X\subset \Sch_X$ be the full subcategory consisting of quasi-separated $X$-schemes of finite type whose fibers over the generic points of $X$ are finite, and let $\GFin_X^\fp=\GFin_X\cap\Sch_X^\fp$. Note that $\GFin_X$ and $\GFin_X^\fp$ are closed under finite sums and finite limits in $\Sch_X$.
 We define the Zariski, Nisnevich, cdp, rh, and cdh topologies on $\GFin_X^\fp$ using the usual cd-structures. 

\begin{lem}\label{lem:proBX}
	$\GFin_X$ is equivalent to the full subcategory of $\Pro(\GFin_X^\fp)$ spanned by limits of cofiltered diagram of closed immersions.
\end{lem} 

\begin{proof}
	We can write any $Y\in\GFin_X$ as the limit of such a diagram $\{Y_i\}$ in $\Sch_X^\fp$ (since $Y$ is quasi-separated of finite type \cite[Tag 09ZQ]{stacks}). Since $(Y_i)_\eta$ is noetherian for every $\eta\in X^\mathrm{gen}$, the closed immersion $Y_\eta\to (Y_i)_\eta$ is eventually an isomorphism, and hence $Y_i\in \GFin_X^\fp$.
\end{proof}

We can extend any presheaf $\sF$ on $\GFin_X^\fp$ to $\Pro(\GFin_X^\fp)$ in the usual way (i.e., by left Kan extension), and in particular, by Lemma~\ref{lem:proBX}, to $\GFin_X$. We shall still denote by $\sF$ this extension. If $\sF$ is a sheaf for the Zariski or Nisnevich topology, its extension to $\GFin_X$ is as well. More interestingly, if $\sF$ is a sheaf for the cdp topology, its extension to $\GFin_X$ satisfies excision for arbitrary abstract blowup squares (see Remark~\ref{rem:abu}).
 
 \begin{lem}\label{lem:cocontinuous}
 	Let $\tau\in\{\mathord\sqcup,\Zar,\Nis,\cdp,\rh,\cdh\}$. Then the restriction functor $\PSh(\Sch_X^\fp) \to \PSh(\GFin_X^\fp)$ commutes with $\tau$-sheafification.
 \end{lem}

\begin{proof}
	It is obvious that the restriction preserves sheaves, so it remains to prove the following: for any $Y\in\GFin_X^\fp$ and any $\tau$-covering sieve $R$ on $Y$ in $\Sch_X^\fp$ induced by a $\tau$-square, the restriction of $R$ to $\GFin_X^\fp$ is $\tau$-covering. Since $\GFin_X^\fp$ is closed under étale extensions, it suffices to consider an abstract blowup square $Z\hook Y\leftarrow Y'$. Let $Y_0'$ be the schematic closure of $Y'\times_Y(Y-Z)$ in $Y'$. By Lemma~\ref{lem:proBX}, there exists a closed subscheme $Y_1'\subset Y'$ containing $Y_0'$ and belonging to $\GFin_X^\fp$. Then $Z\hook Y\leftarrow Y_1'$ is an abstract blowup square in $\GFin_X^\fp$ refining the given square.
\end{proof}

The following definition is a variation on the notion of ``clean sheaf'' introduced by Goodwillie and Lichtenbaum in \cite[Definition 5.1]{goodwillie-lichtenbaum}.

\begin{defn}
	A morphism $Y'\to Y$ in $\GFin_X$ is \emph{clean} if it is proper and induces an isomorphism $(Y'_\eta)_\red\simeq (Y_\eta)_\red$ for every $\eta\in X^\mathrm{gen}$.
	A presheaf on $\GFin_X^\fp$ is \emph{clean} if it sends clean morphisms to equivalences.
\end{defn}

	We denote by $\PSh^\cl(\GFin_X^\fp)\subset \PSh(\GFin_X^\fp)$ the full subcategory of clean presheaves. Note that every clean morphism in $\GFin_X$ is a cofiltered limit of clean morphisms in $\GFin_X^\fp$, so the extension of a clean presheaf to $\GFin_X$ sends all clean morphisms to equivalences.

\begin{rem}
	If $Y\in\GFin_X$ has no narrow components, the category of clean $Y$-schemes is precisely the category $\sN_Y$ defined in \sectsign\ref{sub:RZ}.
\end{rem}

\begin{rem}
	Nilimmersions are clean. Hence, clean presheaves are nilinvariant.
\end{rem}

\begin{rem}\label{rem:cleanSigma}
	If $\sF\in\PSh_\Sigma(\GFin_X^\fp)$ is clean and $Y\in\GFin_X$, then
	\[
	\sF(Y) \simeq \sF(Y_1)\times\dotsb\times\sF(Y_n)
	\]
	where $Y_1,\dotsc,Y_n$ are the irreducible components of $Y$ that dominate an irreducible component of $X$ (with reduced scheme structure).
	Indeed, the morphism $Y_1\sqcup\dotsb\sqcup Y_n\to Y$ is clean.
	In particular, if $Y$ is narrow, then $\sF(Y)$ is contractible.
\end{rem}

\begin{prop}\label{prop:clean=>cdp}
	Let $\sF$ be a clean $\Sigma$-presheaf on $\GFin_X^\fp$. Then $\sF$ is a cdp sheaf.
\end{prop}

\begin{proof}
	We must show that $\sF$ takes every abstract blowup square
	\[
	\begin{tikzcd}
		Z' \ar[hook]{r} \ar{d} & Y' \ar{d} \\
		Z \ar[hook]{r} & Y
	\end{tikzcd}
	\]
	to a cartesian square.
	By Remark~\ref{rem:cleanSigma}, we can assume that $Y$ is integral and dominant (over a component of $X$). If $Z$ is also dominant, then the horizontal maps are isomorphisms. Otherwise, $Z$ is narrow and both vertical morphisms are clean, so $\sF$ takes them to equivalences.
\end{proof}

The class of clean morphisms is stable under base change and composition. 
It follows that the inclusion $\PSh^\cl(\GFin_X^\fp)\subset \PSh(\GFin_X^\fp)$ admits a left adjoint $\cl$ given by
\[
\cl(\sF)(Y) = \colim_{Y'\to Y} \sF(Y'),
\]
where $Y'\to Y$ ranges over clean finitely presented $Y$-schemes \cite[Proposition 3.4(1)]{hoyois-sixops}. Note that this indexing category has finite limits and hence is cofiltered, so that $\cl$ is a left exact localization.

\begin{lem}\label{lem:cl-technical}
	Let $\sF\in\PSh(\GFin_X^\fp)$ and $Y\in \GFin_X$. Let $Z\subset Y$ be the union of the irreducible components of $Y$ that dominate an irreducible component of $X$ (with reduced scheme structure). Then
	\[
	\cl(\sF)(Y)=\colim_{Z'\in\sM_Z} \sF(Z').
	\]
\end{lem}

\begin{proof}
	There exists a clean closed immersion $Y\hook Y'$ with $Y'\in\GFin_X^\fp$, so we may assume that $Y$ is finitely presented.
	By Lemma~\ref{lem:RZcategories}(2,3), we have
	\[
	\colim_{Z'\in \sM_{Z}} \sF(Z') \simeq \colim_{Z'\in \sN_Z^\fp} \sF(Z').
	\]
	Write $Z$ as a cofiltered limit of finitely presented closed subschemes $Z_i\subset Y$. Then $\sN_Z^\fp$ is the filtered colimit of the categories of clean finitely presented $Z_i$-schemes, hence \cite[Corollary 4.2.3.10]{HTT}
	\[
	\colim_{Z'\in \sN_Z^\fp} \sF(Z') = \colim_i \colim_{Z_i'\to Z_i} \sF(Z_i') = \colim_i \cl(\sF)(Z_i),
	\]
	where $Z_i'$ ranges over clean finitely presented $Z_i$-schemes. Since $Z_i\to Y$ is clean, the last colimit is constant and equal to $\cl(\sF)(Y)$.
\end{proof}

\begin{prop}\label{prop:key}
	The functor $\cl\colon \PSh(\GFin_X^\fp)\to\PSh^\cl(\GFin_X^\fp)$ sends $\Sigma$-presheaves to cdp sheaves, Zariski sheaves to rh sheaves, and Nisnevich sheaves to cdh sheaves.
\end{prop}

\begin{proof}
	It is clear that $\cl$ preserves $\Sigma$-presheaves, so the first statement follows from Proposition~\ref{prop:clean=>cdp}.
	Let $\sF$ be a Zariski or Nisnevich sheaf on $\GFin_X^\fp$. It remains to show that $\cl(\sF)$ takes every Zariski or Nisnevich square
	\[
	\begin{tikzcd}
		W \ar[hook]{r} \ar{d} & V \ar{d} \\
		U \ar[hook]{r} & Y
	\end{tikzcd}
	\]
	to a cartesian square. We can assume that $V\to Y$ is separated (since this restricted cd-structure generates the same topology). 
	Then the result follows immediately from Lemma~\ref{lem:cl-technical} and Proposition~\ref{prop:coinitial}.
\end{proof}

For $X$ a quasi-integral qcqs scheme, we define functors
\begin{align*}
	\cl_\cdp &\colon \Shv_\cdp(\Sch_X^\fp) \to \Fun(X^\mathrm{gen},\Spc),\\
	\cl_\rh &\colon \Shv_\rh(\Sch_X^\fp) \to \Shv(\RZ(X)_\Zar),\\
	\cl_\cdh &\colon \Shv_\cdh(\Sch_X^\fp) \to \Shv(\RZ(X)_\Nis).
\end{align*}
By Lemma~\ref{lem:RZcategories}(2,3), $\Shv(\RZ(X)_\Nis)$ is the limit of the $\infty$-topoi $\Shv(Y_\Nis)$ as $Y$ ranges over the category $\sN_X^\fp$ of clean finitely presented $X$-schemes. 
We denote by $\Shv^\lax(\RZ(X)_\Nis)$ the right-lax limit of this diagram, i.e., the $\infty$-category of sections of the corresponding cartesian fibration over $(\sN_X^\fp)^\op$, so that $\Shv(\RZ(X)_\Nis)\subset \Shv^\lax(\RZ(X)_\Nis)$ is the full subcategory of cartesian sections. There is a functor
\[
\rho^\lax\colon \Shv_\Nis(\GFin_X^\fp)\to  \Shv^\lax(\RZ(X)_\Nis), \quad \sF\mapsto (\sF_Y)_{Y\in\sN_X^\fp},
\]
where $\sF_Y$ is the restriction of $\sF$ to the small Nisnevich site of $Y$. If $\sF$ is clean, then the section $\rho^\lax(\sF)$ is cartesian, so $\rho^\lax$ restricts to a functor
\[
\rho\colon \Shv_\Nis^\cl(\GFin_X^\fp) \to  \Shv(\RZ(X)_\Nis).
\]
 The functor $\cl_\cdh$ is then the composition
\[
\Shv_\cdh(\Sch_X^\fp) \to \Shv_\cdh(\GFin_X^\fp) \xrightarrow{\cl} \Shv_\cdh^\cl(\GFin_X^\fp) = \Shv_\Nis^\cl(\GFin_X^\fp) \xrightarrow{\rho} \Shv(\RZ(X)_\Nis),
\]
where the first functor is restriction, $\cl$ preserves cdh sheaves by Proposition~\ref{prop:key}, and the equality is Proposition~\ref{prop:clean=>cdp}.
The functors $\cl_\rh$ and $\cl_\cdp$ are defined similarly, replacing Nisnevich sheaves by Zariski sheaves and $\Sigma$-presheaves, respectively.

\begin{prop}\label{prop:geometric}
	The functors $\cl_\cdp$, $\cl_\rh$, and $\cl_\cdh$ preserve colimits and finite limits.
\end{prop}

\begin{proof}
	We prove the result for $\cl_\cdh$, but the same proof applies to $\cl_\cdp$ and $\cl_\rh$ with obvious modifications.
	The restriction functor $\Shv_\cdh(\Sch_X^\fp) \to \Shv_\cdh(\GFin_X^\fp)$ clearly preserves limits, and it preserves colimits by Lemma~\ref{lem:cocontinuous}.
	The functor $\cl\colon \Shv_\cdh(\GFin_X^\fp)\to \Shv_\cdh^\cl(\GFin_X^\fp)$ clearly preserves finite limits, and it preserves colimits because it is left adjoint to the inclusion. It remains to show that the functor $\rho$ preserves colimits and finite limits. Note that $\rho^\lax$ preserves limits and colimits, so in fact $\rho$ preserves limits. The inclusion $\Shv(\RZ(X)_\Nis)\subset \Shv^\lax(\RZ(X)_\Nis)$ has a left exact left adjoint $\sG\mapsto\sG^\dagger$ given by the formula
	\[
	(\sG^\dagger)_Y = \colim_{p\colon Y'\to Y} p_*(\sG_{Y'})
	\]
	(using the fact that Nisnevich sheaves are closed under filtered colimits).
	If $\sF\in\Shv_\Nis(\GFin_X^\fp)$, we see that $\rho^\lax(\sF)^\dagger \to \rho^\lax(\cl(\sF))^\dagger$ is an equivalence. It follows that the square
	\[
	\begin{tikzcd}
		\Shv_\Nis(\GFin_X^\fp) \ar{r}{\rho^\lax} \ar{d}[swap]{\cl} & \Shv^\lax(\RZ(X)_\Nis) \ar{d}{\dagger} \\
			\Shv_\Nis^\cl(\GFin_X^\fp) \ar{r}{\rho} & \Shv(\RZ(X)_\Nis)\rlap,
	\end{tikzcd}
	\] 
	commutes, hence that $\rho$ preserves colimits.
\end{proof}

% \begin{lem}\label{lem:nil}
% 	Let $X$ be a qcqs scheme with finitely many irreducible components and let $\tau\in\{\cdp,\rh,\cdh\}$. Then the pullback functor $\GFin^\fp_X\to \GFin^\fp_{X_\red}$ induces an equivalence of $\infty$-topoi $\Shv_\tau(\GFin^\fp_{X})\simeq \Shv_\tau(\GFin^\fp_{X_\red})$.
% \end{lem}
%
% \begin{proof}
% 	Write $X_\red$ as a cofiltered limit of finitely presented closed subschemes $X_i\subset X$. Then $\GFin^\fp_{X_\red}\simeq \colim_i\GFin^\fp_{X_i}$. Moreover, every Zariski, Nisnevich, or abstract blowup square in $\GFin^\fp_{X_\red}$ is the pullback of one in $\GFin^\fp_{X_i}$ for some $i$, so that $\Shv_\tau(\GFin^\fp_{X_\red})=\lim_i\Shv_\tau(\GFin^\fp_{X_i})$.
% 	 Since cdp sheaves are nilinvariant, it is clear that each pullback functor $\GFin^\fp_X\to \GFin^\fp_{X_i}$ induces an equivalence between the $\infty$-categories of $\tau$-sheaves, whence the result.
% \end{proof}

\begin{lem}\label{lem:X^0}
	For any qcqs scheme $X$, the topological space $X^\mathrm{gen}$ is Hausdorff and has a basis of clopen sets.
\end{lem}

\begin{proof}
	If $X$ is affine, this is \cite[Corollary 2.4]{HJ}. The general case follows since $X$ is homeomorphic to an affine scheme.
\end{proof}

\begin{thm}\label{thm:hodim-cdh}
	Let $X$ be a qcqs scheme of finite valuative dimension $d$. Then the $\infty$-topoi $\Shv_\cdp(\Sch_X^\fp)$, $\Shv_\rh(\Sch_X^\fp)$, and $\Shv_\cdh(\Sch_X^\fp)$ have homotopy dimension $\leq d$.
\end{thm}

\begin{proof}
	 We prove the theorem by induction on $d$. If $d<0$ then $X$ is empty and the result is trivial.
	Let $\tau\in\{\cdp,\rh,\cdh\}$ and let $\sF$ be a $d$-connective $\tau$-sheaf on $\Sch_X^\fp$. We must show that $\sF(X)$ is nonempty.
	
	Suppose first that $X$ is quasi-integral.
	By Proposition~\ref{prop:geometric} and \cite[Proposition 6.5.1.16(4)]{HTT}, $\cl_\tau(\sF)$ is $d$-connective. By Proposition~\ref{prop:hodimRZ}, $\cl_\tau(\sF)$ admits a global section. 
	The space of global sections of $\cl_\tau(\sF)$ is the filtered colimit of $\sF(Y)$ as $Y$ ranges over modifications of $X$ (Lemma~\ref{lem:cl-technical}). Hence, there exists a modification $f\colon Y\to X$ such that $\sF(Y)$ is nonempty. 
	Let $Z\subset X$ be a narrow closed subscheme such that $f^{-1}(X- Z)\to X-Z$ is an isomorphism. Then $Y$ and $Z$ form an abstract blowup square over $X$, so
	\[
	\sF(X) \simeq \sF(Y)\times_{\sF(Y_Z)} \sF(Z).
	\]
	Since $Z$ and $Y_Z$ have valuative dimension $<d$ by Proposition~\ref{prop:dim_v}(4,6), the induction hypothesis applies and we deduce that $\sF(Z)$ and $\sF(Y_Z)$ are connected. Hence, $\sF(X)$ is nonempty.
	
	We now prove the general case. Let $\sE$ be the poset of closed subschemes $Z\subset X$ such that $\sF(Z)=\varnothing$. If $\{Z_i\}$ is a cofiltered diagram in $\sE$, then $\lim_i Z_i$ belongs to $\sE$ since $\sF(\lim_iZ_i)\simeq\colim_i\sF(Z_i)$. In particular, any chain in $\sE$ admits a lower bound. By Zorn's lemma, it remains to show that $\sE$ does not have a minimal element. If $Z\in\sE$, then $Z^\mathrm{gen}$ has infinitely many points (by the first part of the proof). Using Lemma~\ref{lem:X^0}, we can find a coproduct decomposition $Z^\mathrm{gen}=S_1\sqcup S_2$ with $S_1$ and $S_2$ nonempty. Let $Z_i$ be the closure of $S_i$ in $Z$ with reduced scheme structure. Since $S_i$ is closed in $Z^\mathrm{gen}$, $Z_i\cap Z^\mathrm{gen}=S_i$. In particular, $Z_1\cap Z_2$ is a narrow subscheme of $Z$. By Proposition~\ref{prop:dim_v}(4), $Z_1\cap Z_2$ has valuative dimension $<d$, so $\sF(Z_1\cap Z_2)$ is connected by the induction hypothesis. Since $\varnothing=\sF(Z)\simeq \sF(Z_1)\times_{\sF(Z_1\cap Z_2)}\sF(Z_2)$, we deduce that $\sF(Z_1)=\varnothing$ or $\sF(Z_2)=\varnothing$, and hence that $Z$ is not minimal in $\sE$.
\end{proof}

\begin{cor}\label{cor:cdh-hypercomplete}
	Let $X$ be a qcqs scheme of finite valuative dimension. Let $\tau\in\{\cdp,\rh,\cdh\}$. Then every sheaf in $\Shv_{\tau}(\Sch_X^\fp)$ is the limit of its Postnikov tower. In particular, $\Shv_{\tau}(\Sch_X^\fp)$ is hypercomplete.
\end{cor}

\begin{proof}
	By Corollary~\ref{cor:dim_v}, every $Y\in\Sch_X^\fp$ has finite valuative dimension. Thus, it follows from Theorem~\ref{thm:hodim-cdh} that $\Shv_{\tau}(\Sch_X^\fp)$ is locally of finite homotopy dimension. This implies the result by Proposition~\ref{prop:hodim}(2).
\end{proof}

\begin{rem}
	Corollary~\ref{cor:cdh-hypercomplete} does not assert nor imply that $\Shv_\tau(\Sch_X^\fp)$ is Postnikov complete in the sense of \cite[Definition A.7.2.1]{SAG}. We do not know if this is the case.
\end{rem}

\begin{rem}
	For $\tau\in\{\rh,\cdh\}$ and $X$ any scheme locally of finite valuative dimension, it follows from Corollary~\ref{cor:cdh-hypercomplete} that every $\tau$-sheaf on $\Sch_X^\mathrm{lfp}$ is the limit of its Postnikov tower.
\end{rem}

\begin{cor}\label{cor:stalks-vdim}
	Let $X$ be a scheme locally of finite valuative dimension. Let $f$ be a morphism between finitary rh (resp.\ cdh) sheaves on $\Sch_X$. Then $f$ is an equivalence if and only if it is an equivalence on valuation rings (resp.\ on henselian valuation rings) of finite rank.
\end{cor}

\begin{proof}
	If $f$ is an equivalence on (henselian) valuation rings of finite rank, then it is an equivalence on all (henselian) valuation rings by Lemma~\ref{lem:finite-rank}.
	For $\tau\in\{\rh,\cdh\}$, recall that $\Shv_{\tau}^\fin(\Sch_X)$ is a locally coherent $\infty$-topos (Propositions \ref{prop:finitary-sheaves} and \ref{prop:coherent}), whose points are given by evaluating on (henselian) valuation rings. It thus follows from Deligne's completeness theorem \cite[Theorem A.4.0.5]{SAG} that $f$ is an $\infty$-connective morphism.
	But this $\infty$-topos is hypercomplete by Corollary~\ref{cor:cdh-hypercomplete}, so $f$ is an equivalence.
\end{proof}

\section{Milnor excision and cdarc descent}

\subsection{The cdarc topology}

A qcqs morphism of schemes $Y\to X$ is called a \emph{v~cover} if, for every valuation ring $V$ and every morphism $\Spec(V) \to X$, there exists an extension of valuation rings $V\subset W$ and a morphism $\Spec(W)\to Y$ making the following square commute:
\[
\begin{tikzcd}
	\Spec(W) \ar{r} \ar{d} & Y \ar{d} \\
	\Spec(V) \ar{r} & X\rlap.
\end{tikzcd}
\]
This class of morphisms was introduced by Rydh in \cite{Rydh-submerse} as a better-behaved replacement for the class of universally submersive morphisms (with which it coincides on noetherian schemes)\footnote{More precisely, Rydh defines \emph{universally subtrusive} morphisms, which are not necessarily qcqs. A v~cover is thus a qcqs universally subtrusive morphism.}. Rydh also defines the \emph{h~topology} on the category of schemes to be the topology generated by open covers and v~covers of finite presentation, extending Voevodsky's h~topology \cite{HS1} to nonnoetherian schemes.
In \cite[\sectsign 2]{bhatt-scholze-proj}, Bhatt and Scholze further consider the \emph{v~topology} on the category of schemes, which is the topology generated by open covers and v~covers (with no finite presentability condition). 
Finally, in \cite{arc}, Bhatt and Mathew introduce \emph{arc covers} and the associated \emph{arc topology}: a qcqs morphism of schemes is an arc cover if it satisfies the above definition of a v~cover with $V$ and $W$ valuation rings of rank $\leq 1$.

The rh and cdh topologies are ``completely decomposed'' versions of the h topology.
We next introduce analogous completely decomposed versions of the v topology and the arc topology.

\begin{defn}\label{defn:cdarc}
	A qcqs morphism of schemes $Y\to X$ is:
	\begin{itemize}
		\item an \emph{rv cover} if it is surjective on valuation rings, i.e., for every valuation ring $R$, the induced map
	\[
	\Maps(\Spec R,Y) \to \Maps(\Spec R,X)
	\]
	is surjective;
	\item  a \emph{cdv cover} if it is surjective on henselian valuation rings;
	\item an \emph{rarc cover} if it is surjective on valuation rings of rank $\leq 1$;
	\item a \emph{cdarc cover} if it is surjective on henselian valuation rings of rank $\leq 1$.
	\end{itemize}
\end{defn}

Each of these notions of cover defines a topology on the category of schemes, which is generated by these covers and the Zariski topology. A presheaf $\Sch_S^\op\to \Spc$ is a sheaf for this topology if and only if it is a Zariski sheaf and it satisfies descent with respect to any cover $Y\to X$ between qcqs schemes.

\begin{rem}\label{rem:cdv-vs-cdh}
	An rv cover of finite presentation is an rh cover, because valuation rings are precisely the points of the locally coherent $\infty$-topos $\Shv_\rh(\Sch_X^\mathrm{lfp})$. Similarly, a cdv cover of finite presentation is a cdh cover. 
	Recall that the analogous statement for v covers holds by definition (see \cite[Definition 2.7]{bhatt-scholze-proj}).
\end{rem}

\begin{prop} \label{prop:coarse} 
	Every cdv cover is a v cover and every cdarc cover is an arc cover.
\end{prop}

\begin{proof}
Let $f\colon Y \rightarrow X$ be a cdv cover and let $g\colon\Spec V \rightarrow X$ be a morphism where $V$ is a valuation ring. Consider the extension of valuation rings $V\subset V^h$. By assumption we have a lift $\Spec V^h \rightarrow Y$ over $g$, verifying that $f$ is a v cover. 
If $f$ is a cdarc cover and $V$ has rank $\leq 1$, then $V^h$ has rank $\leq 1$ \cite[Tag 0ASK]{stacks}, so $f$ is an arc cover.
\end{proof}

The following diagram summarizes the relations between these topologies:
\[
\begin{tikzcd}
	\Zar \ar{r} \ar{d} & \rh \ar{r} \ar{d} & \rv \ar{r} \ar{d} & \rarc \ar{d} \\
	\Nis \ar{r} \ar{d} & \cdh \ar{r} \ar{d} & \cdv \ar{r} \ar{d} & \cdarc \ar{d} \\
	\mathrm{fppf} \ar{r} & \hh \ar{r} & \vv \ar{r} & \arc\rlap.
\end{tikzcd}
\]

It follows from Remark~\ref{rem:cdv-vs-cdh} that every cdv cover $f\colon Y\to X$ between qcqs schemes is a cofiltered limit of cdh covers $f_\alpha\colon Y_\alpha\to X$. However, this does not imply that every finitary cdh sheaf satisfies cdv descent, because filtered colimits need not commute with cosimplicial limits. We now introduce a convenient condition on an $\infty$-category $\sC$ ensuring that finitary $\sC$-valued cdh sheaves are cdv sheaves (see Proposition~\ref{prop:stalks}).

\begin{defn}
	An $\infty$-category $\sC$ is \emph{compactly generated by cotruncated objects} if it is compactly generated and every compact object is cotruncated (i.e., truncated in $\sC^\op$).
\end{defn}

\begin{exam}
	If $\sC$ is a compactly generated $\infty$-category, the subcategory $\sC_{\leq n}\subset \sC$ of $n$-truncated objects is compactly generated by cotruncated objects.
\end{exam}

\begin{exam}
	If $\sC$ is a compactly generated $\infty$-category, then the $\infty$-category $\Stab(\sC)_{\leq 0}$ of coconnective spectra in $\sC$ is compactly generated by cotruncated objects.
\end{exam}

\begin{lem}\label{lem:cotruncated}
	Let $\sC$ be an $\infty$-category compactly generated by cotruncated objects.
	\begin{enumerate}
		\item Filtered colimits commute with cosimplicial limits in $\sC$.
		\item If $\sX$ is an $\infty$-topos, any sheaf $\sX^\op\to\sC$ is hypercomplete.\footnote{This statement only requires $\sC$ to be generated under colimits by cotruncated objects.}
	\end{enumerate}
\end{lem}

\begin{proof}
	Both statements follow immediately from the special case $\sC=\Spc_{\leq n}$.
\end{proof}

The following proposition generalizes and gives a different proof of \cite[Proposition 3.31]{arc}.

\begin{prop}\label{prop:stalks}
	Let $S$ be a scheme and let $\sC$ be compactly generated by cotruncated objects.
	\begin{enumerate}
		\item If $\sF\colon \Sch_S^\op\to\sC$ is finitary, then $\sF$ is an rv (resp.\ cdv, v) sheaf if and only if it is an rh (resp.\ cdh, h) sheaf.
		\item Let $\sF,\sG\colon\Sch_S^\op\to\sC$ be finitary rh (resp.\ cdh, h) sheaves. Then a morphism $\sF\to\sG$ is an equivalence if and only if $\sF(V)\to\sG(V)$ is an equivalence for every valuation ring $V$ (resp.\ every henselian valuation ring $V$, every absolutely integrally closed valuation ring $V$).
	\end{enumerate}
\end{prop}

\begin{proof}
	We give the proof for rh, but the same proof works for cdh and h.
	
	(1) By Remark~\ref{rem:cdv-vs-cdh}, every rv cover $f\colon Y\to X$ between qcqs schemes is a cofiltered limit of rh covers $f_\alpha\colon Y_\alpha\to X$, hence the Čech nerve of $f$ is the limit of the Čech nerves of $f_\alpha$. The claim now follows from Lemma~\ref{lem:cotruncated}(1).
	
	(2) Applying $\Maps(c,-)$ for $c\in\sC$ compact, we can assume that $\sC=\Spc_{\leq n}$. Then the result follows from Deligne's completeness theorem \cite[Theorem A.4.0.5]{SAG}, since valuation rings are the points of the locally coherent $\infty$-topos $\Shv_\rh(\Sch_S^\mathrm{lfp})$.
\end{proof}

\subsection{Milnor squares}

A commutative square of schemes
\begin{equation*}\label{eqn:Milnor}
\begin{tikzcd}
	W \ar[hook]{r} \ar{d} & Y \ar{d}{f} \\
	Z \ar[hook]{r}{i} & X 
\end{tikzcd}
\end{equation*}
is called a \emph{Milnor square} if it is bicartesian, $f$ is affine, and $i$ is a closed immersion. We will also denote by $f\colon (Y,W)\to(X,Z)$ or $Q\colon (Y,W)\to(X,Z)$ such a square.

	A Milnor square is also cocartesian in the category of ringed spaces. Indeed, since $Z$ and $Y$ are affine over $X$, the pushout of $Z\leftarrow W\to Y$ in the category of ringed spaces (equivalently, of locally ringed spaces) is a scheme by \cite[Théorème 7.1]{ferrand}. In particular, $f\colon Y\to X$ is an isomorphism over the open complement of $Z$. By \cite[Théorème 5.1]{ferrand}, a Milnor square of affine schemes is equivalently a cartesian square of commutative rings
\[
\begin{tikzcd}
	A \ar{r} \ar{d} & B \ar{d} \\
	A/I \ar{r} & B/J\rlap.
\end{tikzcd}
\]
A general Milnor square is therefore Zariski-locally of this form.

\begin{rem}\label{rem:Milnor-spans}
	A span of schemes $Z\stackrel h\leftarrow W\stackrel e\to Y$ with $h$ affine and $e$ a closed immersion can be completed to a Milnor square if and only if there exist affine morphisms $a\colon Y\to S$ and $b\colon Z\to S$ such that $a\circ e = b\circ h$.
\end{rem}

\begin{defn}
	A presheaf $\sF\colon \Sch_S^\op\to\sC$ satisfies \emph{Milnor excision} if $\sF(\varnothing)=*$ and $\sF$ sends Milnor squares to cartesian squares.
\end{defn}

Note that if $\sF$ is a Zariski sheaf, then $\sF$ satisfies Milnor excision if and only if it sends Milnor squares of affine schemes to pullback squares.

\begin{exam}
	For any $S$-scheme $X$, the representable presheaf $\Maps_S(-,X)$ satisfies Milnor excision, since Milnor squares are cocartesian.
\end{exam}

\begin{exam}
	Homotopy K-theory $\KH\colon (\Sch^\qcqs)^\op\to\Spt$ satisfies Milnor excision \cite[Theorem 2.1]{Weibel}.
\end{exam}

\begin{rem}
	If $\sC$ is a pointed $\infty$-category with finite products and $\sF\colon \Sch_R^\op\to \sC$ is a $\Sigma$-presheaf, one can extend $\sF$ to nonunital commutative $R$-algebras by setting
	\[
	\sF(I) = \fib(\sF(R\times I)\to \sF(R)),
	\]
	where $R\times I$ is the unitalization of the nonunital $R$-algebra $I$. If $\sF$ satisfies Milnor excision, then for any $R$-algebra $A$ and any ideal $I\subset A$, the null sequence
	\[
	\sF(I) \to \sF(A) \to \sF(A/I)
	\]
	is a fiber sequence. In fact, if $\sF$ is a Zariski sheaf and $\sC^\op$ is prestable, this property is equivalent to Milnor excision.
\end{rem}

\begin{lem}[Approximating Milnor squares]
	\label{lem:Milnor-approx}
	Let $S$ be a qcqs scheme.
	\begin{enumerate}
		\item Every Milnor square $(Y,W)\to (X,Z)$ in $\Sch_S^\qcqs$ is a cofiltered limit of Milnor squares $(Y_\alpha,W_\alpha)\to (X_\alpha,Z_\alpha)$ where $Z_\alpha$, $Y_\alpha$, and $W_\alpha$ are of finite presentation over $S$. If $X$ is affine over $S$, we can assume $X_\alpha$ affine over $S$.
		\item Every Milnor square $f\colon(Y,W)\to (X,Z)$ in $\Sch_S^\qcqs$ is a cofiltered limit of Milnor squares $f_\alpha\colon (Y_\alpha,W_\alpha)\to (X,Z)$ where $f_\alpha$ is of finite type.
	\end{enumerate}
\end{lem}

\begin{proof}
	(1) Assume first that $X$ is affine over $S$. In what follows we implicitly use Remark~\ref{rem:Milnor-spans} to obtain Milnor squares from spans of affine $S$-schemes.
	 Since $X$ is qcqs, we can write $Z$ as a cofiltered intersection of finitely presented closed subschemes $Z_\alpha\subset X$. Then the span $Z\leftarrow W\hook Y$ is the limit of the spans $Z_\alpha\leftarrow Y\times_XZ_\alpha \hook Y$, so we can assume that $W\hook Y$ is finitely presented. Writing $Z$ as a limit of finitely presented affine $S$-schemes, we can also assume that $Z$ is finitely presented over $S$. Since $W\hook Y$ is finitely presented, we can write it as a limit of closed immersions $W_\alpha\hook Y_\alpha$ between finitely presented affine $S$-schemes, and since $Z$ is finitely presented the map $W\to Z$ factors through some $W_\alpha$.
	
	In particular, this proves the result when $S=X$. In general, we can write $X$ as a limit of finitely presented $S$-schemes $X_\alpha$, and any span of finitely presented affine $X$-schemes is pulled back from $X_\alpha$ for some $\alpha$.
	
(2) The $\sO_X$-algebra $\sA=f_*(\sO_Y)$ is the filtered union of its finitely generated subalgebras $\sA_\alpha$. Set $Y_\alpha=\Spec\sA_\alpha$ and $W_\alpha=Y_\alpha\times_XZ$. If $\sI\subset\sO_X$ is the ideal of $Z$, then $\sI\to\sI\sA_\alpha$ is an isomorphism, since the composite $\sI\to\sI\sA_\alpha\hook \sI \sA$ is. This shows that $(Y_\alpha,W_\alpha)\to (X,Z)$ is a Milnor square.
\end{proof}

\begin{cor}\label{cor:Milnor-approx}
	Let $S$ be a qcqs scheme of finite valuative dimension. Then every Milnor square $f\colon(Y,W)\to (X,Z)$ in $\Sch_S^\qcqs$ is a cofiltered limit of Milnor squares $f_\alpha\colon(Y_\alpha,W_\alpha)\to (X_\alpha,Z_\alpha)$ where $X_\alpha$ has finite valuative dimension and $f_\alpha$ is of finite type.
\end{cor}

\begin{proof}
	By Lemma~\ref{lem:Milnor-approx}(1), we can assume that $Y$ and $Z$ are finitely presented over $S$. 
	The complement of $Z$ in $X$ is isomorphic to $Y-W$, so the degrees of the residual field extensions of $X$ over $S$ are bounded. It follows from Proposition~\ref{prop:residue} that $X$ has finite valuative dimension, and we conclude using Lemma~\ref{lem:Milnor-approx}(2).
\end{proof}

\begin{lem}\label{lem:Milnor-BC}
	Let $Q\colon (Y,W)\to(X,Z)$ be a Milnor square, let $h\colon X'\to X$ be a morphism, and let $Q'\colon (Y',W')\to (X',Z')$ be the base change of $Q$ along $h$. If $h$ is flat or if $X'$ is reduced, then $Q'$ is a Milnor square.
\end{lem}

\begin{proof}
	We only have to check that $Q'$ is cocartesian, i.e., that the square
	\[
	\begin{tikzcd}
		\sO_{X'} \ar{r} \ar[->>]{d} & \sO_{Y'} \ar[->>]{d} \\
		\sO_{Z'} \ar{r} & \sO_{W'}
	\end{tikzcd}
	\]
	of quasi-coherent sheaves on $X'$ is cartesian. If $\sI\subset\sO_X$ is the ideal of $Z$, the kernels of the vertical maps are the images of $h^*(\sI)$. Thus, the above square is cartesian if and only if the restriction of the map $\sO_{X'}\to \sO_{Y'}\times_{\sO_{W'}}\sO_{Z'}$ to the image of $h^*(\sI)$ is injective. If $h$ is flat, then $h^*$ is left exact and so $h^*(\sI)\to \sO_{Y'}$ is injective. If $X'$ is reduced, we claim that $\sO_{X'}\to \sO_{Y'}\times_{\sO_{W'}}\sO_{Z'}$ is injective. It suffices to show that $x^*(\sO_{Y'}\times\sO_{Z'})\neq 0$ for any (generic) point $x\in X'$ \cite[Tag 00EW]{stacks}. But this follows from the fact that $Y'\sqcup Z'\to X'$ is completely decomposed.
\end{proof}

\begin{rem}
	Milnor squares are not closed under arbitrary base change, and so Milnor excision is not (a priori) a topological condition.
	However, it becomes a topological condition when supplemented with nilinvariance. 
	Indeed, let $P$ be the cd-structure on $\Sch_S$ consisting of cartesian squares
	\[
	\begin{tikzcd}
		W \ar[hook]{r} \ar{d} & Y \ar{d}{f} \\
		Z \ar[hook]{r}{i} & X 
	\end{tikzcd}
	\]
	where $i$ is a closed immersion, $f$ is affine, and the induced map $Y\sqcup_WZ \to X$ is a nilimmersion.
	It is clear that a presheaf on $\Sch_S$ satisfies $P$-excision if and only if it is nilinvariant and satisfies Milnor excision. Using Lemma~\ref{lem:Milnor-BC}, it is easy to check that the cd-structure $P$ is stable under base change. In fact, it satisfies Voevodsky's criterion \cite[Theorem 3.2.5]{AHW}, so that a presheaf on $\Sch_S$ satisfies $P$-excision if and only if it is a sheaf for the topology generated by $P$.
\end{rem}

\begin{prop}[Bhatt–Mathew]
	\label{prop:Milnor-rarc}
	Let $(Y,W)\to (X,Z)$ be a Milnor square. Then $Y\sqcup Z\to X$ is an rarc cover.
\end{prop}

\begin{proof}
	By Lemma~\ref{lem:Milnor-BC}, we can assume that $X=\Spec R$ where $R$ is a valuation ring of rank $\leq 1$.
	Then $Z=\Spec R/I$ and $Y=\Spec S$.
	 If $I=0$ or $I=R$, the result is trivial. Otherwise, there exists $x\in I$ that is neither zero nor invertible. Then $R$ has rank $1$ and $R[1/x]$ is the fraction field of $R$. Since $Y\sqcup Z\to X$ is completely decomposed, the map $R\to R[1/x]$ factors through a map $h\colon S\to R[1/x]$. To conclude, we show that $h(S)\subset R$.
	 Since $R$ has rank $1$ and $R\subset h(S)$, either $h(S)=R$ or $h(S)=R[1/x]$. Now $I\to IS$ is an isomorphism, so the restriction of $h$ to $IS$ lands in $R$. In particular $xh(S)\subset R$, which is only possible if $h(S)=R$.
\end{proof}

The following corollary generalizes \cite[Proposition 4.15]{arc}.

\begin{cor}\label{cor:cdarc=>milnor}
	Let $\sC$ be an $\infty$-category and let $\sF\colon \Sch_S^\op\to\sC$ be an rarc sheaf. Then $\sF$ satisfies Milnor excision.
\end{cor}

\begin{proof}
	We can assume $\sC=\Spc$.
	We prove that $\sF$ sends any base change of a Milnor square to a cartesian square. Such squares form a cd-structure on $\Sch_S$ satisfying the assumptions of Voevodsky's criterion \cite[Theorem 3.2.5]{AHW}. The claim then follows from \emph{loc.\ cit.} and Proposition~\ref{prop:Milnor-rarc}.
\end{proof}

\subsection{Criteria for Milnor excision and cdarc descent}

If $V$ is a valuation ring and $\p\subset V$ is a prime ideal, the square
\begin{equation}\label{eqn:v-square}
\begin{tikzcd}
	V \ar{r} \ar{d} & V_\p \ar{d} \\
	V/\p \ar{r} & \kappa(\p)
\end{tikzcd}
\end{equation}
is an example of a Milnor square (see \cite[Lemma 3.12]{huber-kelly} or \cite[Proposition 2.8]{arc}).

\begin{defn}
	A presheaf $\sF$ on $\Sch_S$ satisfies \emph{v-excision} if, for every valuation ring $V$ over $S$ and every prime ideal $\p\subset V$, $\sF$ sends the square~\eqref{eqn:v-square} to a cartesian square. We say that $\sF$ satisfies \emph{henselian v-excision} if this holds whenever $V$ is a henselian valuation ring.
\end{defn}

\begin{rem}
	If $\sF$ is finitary and $S$ is locally of finite valuative dimension, it suffices to check (henselian) v-excision on valuation rings of finite rank (Lemma~\ref{lem:finite-rank}).
\end{rem}

\begin{thm}\label{thm:main}
	Let $S$ be a scheme, $\sC$ a compactly generated $\infty$-category, and $\sF\colon \Sch_S^{\op}\to \sC$ a finitary rh (resp.\ cdh) sheaf. Consider the following assertions:
	\begin{enumerate}
		\item $\sF$ satisfies v-excision (resp.\ henselian v-excision);
		\item $\sF$ satisfies Milnor excision;
		\item $\sF$ satisfies rarc (resp.\ cdarc) descent.
	\end{enumerate}
	In general, $(3)\Rightarrow(2)\Rightarrow(1)$. If $S$ has finite valuative dimension, then $(1)\Rightarrow (2)$. If every compact object of $\sC$ is cotruncated, then $(1)\Rightarrow (3)$.
\end{thm}

The implication $(2)\Rightarrow(1)$ holds because the squares~\eqref{eqn:v-square} are Milnor squares, and the implication $(3)\Rightarrow(2)$ is Corollary~\ref{cor:cdarc=>milnor}. The proofs of $(1)\Rightarrow (2)$ and $(1)\Rightarrow (3)$ will require some preliminaries.

Recall that a proper ideal $I$ in a valuation ring $V$ is prime if and only if it is radical, in which case the quotient $V/I$ is a valuation ring. Moreover, for any subset $S\subset V$ with $0\notin S$, there is a prime ideal $\p$ such that $S^{-1}V=V_\p$, and this is again a valuation ring.

\begin{lem} \label{lem:hens-loc} 
	Let $V$ be a (henselian) valuation ring and $\mathfrak{p}\subset V$ a prime ideal. Then $V/\p$ and $V_\p$ are also (henselian) valuation rings.
\end{lem}

\begin{proof}
	The only nontrivial statement is that $V_\p$ is henselian if $V$ is. This follows from \cite[page 50, Corollary]{nagata-hensel}.
\end{proof}

\begin{lem} \label{lem:extend} 
	Let $V$ be a valuation ring, $\sC$ an $\infty$-category, and $\sF\colon \Sch_V^{\op} \rightarrow \sC$ a nilinvariant presheaf satisfying v-excision (resp.\ henselian v-excision). Then, for any prime ideal $\p\subset V$ and any valuation ring (resp.\ any henselian valuation ring) $W$ over $V$, the square
\[
\begin{tikzcd}
\sF(W) \ar{r} \ar{d} & \sF(W \otimes_V V_{\mathfrak{p}}) \ar{d} \\
\sF(W\otimes_V V/\mathfrak{p}) \ar{r} & \sF(W \otimes_V \kappa(\mathfrak{p}))
\end{tikzcd}
\]
is cartesian.
\end{lem}

\begin{proof}
	By Lemma~\ref{lem:Milnor-BC}, the image of $\Spec W\to \Spec V$ is an interval in the specialization poset. Hence, if $\p$ is not in the image, the result is trivial. Otherwise, let $\q\subset W$ be the maximal prime ideal lying over $\p$, so that $W \otimes_V V_{\mathfrak{p}}\simeq W_\q$. Note that $\sqrt{\p W}\subset\q$.
	 We can then form the commutative diagram
\begin{equation*} \label{eq:big}
\begin{tikzcd}
W \ar{r} \ar{d} & W_{\mathfrak{q}} \ar{r} \ar{d} & W_{\sqrt{\mathfrak{p}W}} \ar{d} \\
W/\mathfrak{p}W \ar{r} \ar{d}  & W_{\mathfrak{q}}/\mathfrak{p}W_{\mathfrak{q}} \ar{r} \ar{d} & W_{\sqrt{\mathfrak{p}W}}/\mathfrak{p}W_{\sqrt{\mathfrak{p}W}} \ar{d}\\
W/\sqrt{\mathfrak{p}W} \ar{r} & W_{\mathfrak{q}}/\sqrt{\mathfrak{p}W_\q} \ar{r} & \kappa(\sqrt{\mathfrak{p}W})\rlap.
\end{tikzcd}
\end{equation*}
Our goal is to prove that $\sF$ converts the top left square into a cartesian square. Both the boundary square and the right rectangle are instances of~\eqref{eqn:v-square}. Moreover, by Lemma~\ref{lem:hens-loc}, $W_\q$ is henselian if $W$ is. By assumption, $\sF$ takes these squares to cartesian squares. It then suffices to prove that $\sF$ takes the bottom left square to a cartesian square. But this follows from the fact that $\sF$ is nilinvariant.
\end{proof}

The next lemma is analogous to \cite[Lemma 4.4]{arc}.

\begin{lem} \label{lem:more-squares} 
	Let $V$ be a valuation ring, $\sC$ a compactly generated $\infty$-category, and $\sF\colon\Sch_V^{\op} \rightarrow \sC$ a finitary rh (resp.\ cdh) sheaf satisfying v-excision (resp.\ henselian v-excision).
	Assume that $V$ has finite rank or that compact objects in $\sC$ are cotruncated.
	For every prime ideal $\p\subset V$ and every $V$-scheme $X$, the square
\[
\begin{tikzcd}
\sF(X) \ar{r} \ar{d} & \sF(X\otimes_V V_\p) \ar{d} \\
\sF(X\otimes_VV/\p) \ar{r} & \sF(X\otimes_V \kappa(\p))
\end{tikzcd}
\]
is cartesian.
\end{lem}

\begin{proof} 
	Consider the canonical map
	\[
	\phi\colon \sF(-)\to \sF(-\otimes_VV_\p) \times_{\sF(-\otimes_V\kappa(\p))} \sF(-\otimes_VV/\p)
	\]
	between finitary rh (resp.\ cdh) sheaves on $\Sch_V$. By Lemma~\ref{lem:extend}, $\phi$ is an equivalence on (henselian) valuation rings. 
	If $V$ has finite rank, we deduce that $\phi$ is an equivalence by Corollary~\ref{cor:stalks-vdim}, applied to $\Maps(c,\sF(-))$ for every compact object $c\in\sC$.
	On the other hand, if compact objects in $\sC$ are cotruncated, then $\phi$ is an equivalence by Proposition~\ref{prop:stalks}(2).
\end{proof}

\begin{lem}\label{lem:milnor-section}
	Let $Q\colon (Y,W)\to (X,Z)$ be a base change of a Milnor square in $\Sch_S$ such that $Y\to X$ admits a section. If $\sC$ is a compactly generated $\infty$-category and $\sF\colon \Sch_S^\op\to \sC$ is a finitary rh sheaf, then $\sF(Q)$ is cartesian.
\end{lem}

\begin{proof}
	We can assume $X$ qcqs.
	We have cartesian morphisms $(X,Z)\to (Y,W)\to (X,Z)$.
		Since $Y\to X$ is an isomorphism over $X-Z$, the closed immersions $X\to Y$ and $W\to Y$ cover $Y$. Writing them as cofiltered limits of finitely presented closed immersions, we deduce that $\sF$ sends the square $(X,Z)\to (Y,W)$ to a cartesian square. Hence, $\sF(Q)$ is also cartesian.
\end{proof}

\begin{proof}[Proof of (1) $\Rightarrow$ (2) in Theorem~\ref{thm:main}]
	Let $Q\colon (Y,W)\to (X,Z)$ be a Milnor square in $\Sch_S$. We want to show that $\sF(Q)$ is cartesian. Since $\sF$ is a Zariski sheaf, we can assume that $S$ and $X$ are qcqs. By Corollary~\ref{cor:Milnor-approx}, since $\sF$ is finitary, we can assume that $X$ has finite valuative dimension. Consider the morphism
	\[
	\phi\colon \sF(-) \to \sF(-\times_XY) \times_{\sF(-\times_XW)} \sF(-\times_XZ)
	\]
	between finitary rh (resp.\ cdh) sheaves on $\Sch_X$. By Corollary~\ref{cor:stalks-vdim}, it remains to show that $\phi$ is an equivalence on $\Spec V$ for every (henselian) valuation ring of finite rank $V$ over $X$. We prove this by induction on the rank of $V$.
	 If $V$ has rank $\leq 1$, then $\Spec V\to X$ lifts to $Y\sqcup Z$ by Proposition~\ref{prop:Milnor-rarc}; if it lifts to $Z$ the result is trivial, otherwise we conclude using Lemma~\ref{lem:milnor-section}. If $V$ has rank $\geq 2$, choose a prime ideal $\p\subset V$ which is neither zero nor maximal. By the induction hypothesis, $\phi$ is an equivalence on $V_\p$, $V/\p$, and $\kappa(\p)$. It then follows from Lemma~\ref{lem:more-squares} that $\phi$ is an equivalence on $V$.
\end{proof}

\begin{proof}[Proof of (1) $\Rightarrow$ (3) in Theorem~\ref{thm:main}]
	The proof is essentially the same as that of \cite[Proposition 4.8]{arc} with minor simplifications. 
	We assume that $\sC$ is compactly generated by cotruncated objects and that $\sF\colon\Sch_S^\op\to\sC$ is a finitary cdh sheaf satisfying henselian v-excision (the rh case is treated similarly). Let $f\colon Y\to X$ be a cdarc cover between qcqs schemes in $\Sch_S$ and let $\check C_\bullet(f)$ be its Čech nerve. We must show that $\sF(X)\to \lim_n\sF(\check C_n(f))$ is an equivalence. We can write $f$ as a cofiltered limit of finitely presented morphisms $f_\alpha\colon Y_\alpha\to X$, which are automatically cdarc covers; using Lemma~\ref{lem:cotruncated}(1) and the assumption that $\sF$ is finitary, we can assume that $f$ is finitely presented.
	 We consider the canonical morphism
	\[
	\phi\colon \sF(-) \to \lim_n\sF(\check C_n(f)\times_X -),
	\]
	which is a morphism between finitary cdh sheaves on $\Sch_X$. By Proposition~\ref{prop:stalks}(2), we are reduced to the case where $X=\Spec V$ for $V$ a henselian valuation ring. 
	
	For $\p\subset \q$ prime ideals in $V$, let $[\p,\q]$ be the set of prime ideals containing $\p$ and contained in $\q$; such a subset of $\Spec V$ will be called an \emph{interval}. We denote by $\mathrm{Int}_V$ the poset of intervals in $\Spec V$. 
	For an interval $I=[\p,\q]$, we let $V_I=(V/\p)_\q$. Then $V_I$ is a henselian valuation ring by Lemma~\ref{lem:hens-loc}, and the assignment $ I \mapsto V_I$ is a functor $\mathrm{Int}_V^\op \to \CAlg_V$ that preserves filtered colimits. Let $\sE\subset \mathrm{Int}_V$ be the subposet of intervals $I$ such that $\phi$ is an equivalence on $V_J$ for all $J\subset I$. The subposet $\sE$ has the following properties:
	 \begin{enumerate}
	 	\item If $f$ has a section over $V_I$, then $I\in\sE$.
	 	\item If $I$ has length $\leq 1$, then $I\in\sE$. This follows from (1) and the definition of cdarc cover.
	 	\item For any prime $\p\neq\m$, there exists $\q>\p$ such that $[\p,\q]\in\sE$. If $\p$ has an immediate successor $\q$, this follows from (2). Otherwise, $\kappa(\p)=\colim_{\q>\p}V_{[\p,\q]}$. Since $f$ is finitely presented and has a section over $\kappa(\p)$, it has a section over $V_{[\p,\q]}$ for some $\q>\p$, so the claim follows from (1).
	 	\item For any prime $\q\neq 0$, there exists $\p<\q$ such that $[\p,\q]\in\sE$. The proof is exactly the same as in (3).
		\item If $[\p,\q]\in\sE$ and $[\q,\mathfrak r]\in\sE$, then $[\p,\mathfrak r]\in\sE$. This follows from Lemma~\ref{lem:more-squares}.
		\item $\sE$ is closed under filtered colimits in $\mathrm{Int}_V$. Indeed, if $[\p,\q]$ is the filtered colimit of $[\p_\alpha,\q_\alpha]\in \sE$, then by (2)–(4) there exists $\alpha$ such that $[\p,\p_\alpha]\in\sE$ and $[\q_\alpha,\q]\in\sE$, hence $[\p,\q]\in\sE$ by (5).
	 \end{enumerate}
	 By (6) and Zorn's lemma, there exists a maximal interval $[\p,\q]\in\sE$. By (3) and (5), we must have $\q=\m$, and by (4) and (5) we must have $\p=0$. This means that $\phi$ is an equivalence on $V$, as desired.
\end{proof}

\begin{rem}
	Under the assumption that $\sC$ is compactly generated by cotruncated objects, Theorem~\ref{thm:main} is analogous to \cite[Theorem 4.1]{arc}, with the h topology replaced by the rh or cdh topology. We remark that there is no obvious version of this result for the eh topology (the topology generated by the étale and cdp topologies), because if $V$ is a strictly henselian valuation ring and $\p\subset V$ is a prime ideal, $V_\p$ is not strictly henselian in general.
\end{rem}

For $k$ a perfect field, let $\Omega^n_\cdh$ denote the cdh sheafification of the presheaf $X\mapsto \Omega^n(X/k)$ (which coincides with its rh sheafification \cite[Theorem 4.12]{huber-kelly}).

\begin{cor}\label{cor:Omega}
	Let $k$ be a perfect field and $n\geq 0$.
	Then $\Omega^n_\cdh\colon \Sch_k^\op\to \mathrm{Mod}_k$ satisfies cdarc descent. In particular, it satisfies Milnor excision.
\end{cor}

\begin{proof}
	By \cite[Lemma 3.14]{huber-kelly}, the cdh sheaf $\Omega^n_\cdh$ satisfies v-excision. It is also finitary since $X\mapsto \Omega^n(X/k)$ is and cdh sheafification preserves this property. Hence, the corollary follows from the implication $(1)\Rightarrow(3)$ of Theorem~\ref{thm:main}.
\end{proof}

% \begin{rem}
% 	Theorem~\ref{thm:main} is a partial converse to \cite[Theorem 3.12]{CHSW}, which states that a nilinvariant Nisnevich sheaf of spectra on $\Sch_k^\fp$, where $k$ is a field of characteristic zero, satisfies cdh descent if it satisfies Milnor excision and sends regular blowup squares to cartesian squares.
% \end{rem}

We shall say that a presheaf $\sF$ on $\Sch_S$ satisfies \emph{rigidity} for a henselian pair $(A,I)$ over $S$ if $\sF(A)\to\sF(A/I)$ is an equivalence.

\begin{thm}\label{thm:rigidity}
	Let $S$ be a scheme, $\sC$ a compactly generated $\infty$-category, and $\sF\colon \Sch_S^{\op}\to \sC$ a finitary cdh sheaf.
	Suppose that $\sF$ satisfies rigidity for henselian valuation rings.
	\begin{enumerate}
		\item If $S$ has finite valuative dimension, then $\sF$ satisfies Milnor excision.
		\item If compact objects in $\sC$ are cotruncated, then $\sF$ is a cdarc sheaf.
	\end{enumerate}
\end{thm}

\begin{proof}
	The assumption and Lemma~\ref{lem:hens-loc} immediately imply that $\sF$ satisfies henselian v-excision. The conclusion therefore follows from Theorem~\ref{thm:main}.
\end{proof}

% \begin{exam}
% 	Let $A$ be a torsion abelian group. Then $C^*_\et(-,A)\colon \Sch^\op \to D(\Z)$ is a finitary cdh sheaf satisfying rigidity for all henselian pairs. Since it lands in $D(\Z)_{\leq 0}$, it follows from Theorem~\ref{thm:rigidity} that $C^*_\et(-,A)$ satisfies cdarc descent. In fact, it satisfies arc descent \cite[]{arc}.
% \end{exam}

\begin{cor}\label{cor:rigidity}
	Let $R$ be a henselian local ring, $\sC$ a compactly generated $\infty$-category such that $\sC^\op$ is prestable, and $\sF\colon \Sch_R^{\op}\to \sC$ a finitary cdh sheaf. Assume that either $R$ has finite valuative dimension or that compact objects in $\sC$ are cotruncated.
	 If $\sF$ satisfies rigidity for henselian local rings, then $\sF$ satisfies rigidity for all henselian pairs.
\end{cor}

\begin{proof}
	Let $(A,I)$ be a henselian pair over $R$, and form the Milnor square
	\[
	\begin{tikzcd}
		R' \ar{r} \ar{d} & A \ar{d} \\
		R \ar{r} & A/I\rlap.
	\end{tikzcd}
	\]
	By Gabber's characterization of henselian pairs \cite[Tag 09XI]{stacks}, $R'$ is a henselian local ring. By Theorem~\ref{thm:rigidity}, $\sF$ satisfies Milnor excision, so we have a cartesian square
	\[
	\begin{tikzcd}
		\sF(R') \ar{r} \ar{d} & \sF(A) \ar{d} \\
		\sF(R) \ar{r} & \sF(A/I)
	\end{tikzcd}
	\]
	in $\sC$. Since $\sC^\op$ is prestable, this square is also cocartesian. By assumption, $\sF$ satisfies rigidity for $R'$ and $R$, which have the same residue field, so $\sF(R')\to \sF(R)$ is an equivalence. It follows that $\sF(A)\to\sF(A/I)$ is an equivalence.
\end{proof}

\subsection{Application to motivic spectra}

In this subsection, we use Theorem~\ref{thm:rigidity} to show that motivic spectra with suitable torsion satisfy Milnor excision.
The main ingredients are the rigidity theorem of Ananyevskiy and Druzhinin \cite{sh-rigid}, the inseparable local desingularization theorem of Temkin \cite{temkin-insep}, and the perfection-invariance theorem of Khan and the first author \cite{shperf}.

Let $S$ be a scheme and $E\in\SH(S)$ a motivic spectrum. We denote by $E(-)\colon \Sch_S^\op\to \Spt$ the associated spectrum-valued cohomology theory: for an $S$-scheme $X$, $E(X)$ is the mapping spectrum from $\1_X$ to $E_X$ in $\SH(X)$.

\begin{lem}\label{lem:E(-)}
	For any scheme $S$ and any $E\in\SH(S)$, $E(-)\colon \Sch_S^\op\to \Spt$ is a finitary cdh sheaf.
\end{lem}

\begin{proof}
	The presheaf $E(-)$ satisfies cdh descent by \cite[Corollary 6.25]{hoyois-sixops}, and it is finitary by \cite[Proposition C.12(4)]{HoyoisGLV}.
\end{proof}

Let $\sC$ be a symmetric monoidal $\infty$-category and $\Phi\subset \pi_0\Maps(\1,\1)$ a multiplicative subset. Recall that an object $E\in\sC$ is \emph{$\Phi$-periodic} if every $\phi\in\Phi$ acts invertibly on $E$ \cite[\sectsign 12.1]{norms}. We say that $E\in\sC$ is \emph{$\Phi$-nilpotent} if $\Maps(E,E')\simeq *$ for every $\Phi$-periodic object $E'$, and \emph{$\Phi$-torsion} if there exists $\phi\in\Phi$ such that $\phi\colon E\to E$ factors through an initial object of $\sC$. Note that every $\Phi$-torsion object is $\Phi$-nilpotent.

\begin{lem}\label{lem:nilpotent}
	Let $\sC$ be a stable compactly generated symmetric monoidal $\infty$-category and let $\Phi\subset \pi_0\Maps(\1,\1)$ be a multiplicative subset. Then the full subcategory of $\Phi$-nilpotent objects in $\sC$ is compactly generated, with compact objects the compact $\Phi$-torsion objects of $\sC$.
\end{lem}

\begin{proof}
	Under the assumption on $\sC$, the full subcategory of $\Phi$-periodic objects is reflective, with localization functor $L_\Phi$ sending $E$ to the colimit of a filtered diagram in which every arrow is of the form $\phi\colon E\to E$ for some $\phi\in\Phi$ \cite[Lemma 12.1]{norms}. If $R_\Phi$ is the fiber of $\id_\sC\to L_\Phi$, then $R_\Phi$ is right adjoint to the inclusion of the full subcategory of $\Phi$-nilpotent objects, and it preserves colimits.
	Let $E\in\sC$ be $\Phi$-nilpotent. Then $E$ is a filtered colimit of objects of the form $R_\Phi(C)$ with $C\in\sC$ compact, and $R_\Phi(C)$ itself is a filtered colimit of objects of the form $\fib(\phi\colon C\to C)$ for $\phi\in\Phi$. It remains to observe that $\fib(\phi\colon C\to C)$ is compact and $\phi^2$-torsion.
\end{proof}

\begin{lem}\label{lem:GW-nilpotence}
	Let $k$ be a field of characteristic $p>0$ and let $\phi\in\GW(k)$ be an element such that $\rk(\phi)$ is prime to $p$. Then $\GW(k)=(p,\phi)$.
\end{lem}

\begin{proof}
	Recall that $h\Z\subset\GW(k)$ is an ideal such that $\GW(k)/h\Z=\W(k)$. Moreover, $\W(k)$ is $2$-power torsion and the invertible elements in $\W(k)$ are forms of odd rank \cite[III \sectsign 3]{milnor1973symmetric}. If $p=2$, then $h=2$ and $\phi$ has odd rank, so $\phi$ becomes invertible in $\W(k)=\GW(k)/2\Z$, which proves the claim.
	Suppose that $p$ is odd. Since $\W(k)$ is $2$-power torsion, $p$ becomes invertible in $\W(k)$, so $(p,h)$ is the unit ideal. On the other hand, since $\phi h=\rk(\phi) h$ and $\rk(\phi)$ is prime to $p$, $h$ belongs to $(p,\phi)$. Hence, $(p,\phi)$ is the unit ideal.
\end{proof}

The following theorem is a minor generalization of the rigidity theorem of Ananyevskiy and Druzhinin \cite{sh-rigid}.

\begin{thm}[Ananyevskiy–Druzhinin]
	\label{thm:smooth-rigidity}
	Let $k$ be a field and let $\Phi\subset\GW(k)$ be the subset of elements with rank invertible in $k$. If $E\in\SH(k)$ is $\Phi$-nilpotent, then $E(-)\colon \Sch_k^\op\to\Spt$ satisfies rigidity for ind-smooth henselian local rings.
\end{thm}

\begin{proof}
	By Lemma~\ref{lem:nilpotent}, we may assume that $E$ is $\Phi$-torsion.
	It follows from Lemma~\ref{lem:GW-nilpotence} that the exponential characteristic of $k$ acts invertibly on $E$, so that $E(-)$ is invariant under universal homeomorphisms \cite[Corollary 2.1.5]{shperf}. Since henselian pairs are preserved by integral base change, we may replace $k$ by $k^\perf$ and hence assume $k$ perfect. 
	Furthermore, since $E(-)$ is finitary, it suffices to prove rigidity for henselizations of points on smooth $k$-schemes.
	In this case the result is precisely \cite[Corollary 7.11]{sh-rigid}.
\end{proof}

If $k$ has characteristic zero, Zariski showed that every valuation ring over $k$ is ind-smooth \cite[Theorem U$_3$]{Zariski} (this is also a consequence of Hironaka's resolution of singularities). Therefore, Theorem~\ref{thm:smooth-rigidity} implies that $\Phi$-nilpotent motivic spectra over $k$ satisfy rigidity for henselian valuation rings. Our next goal is to prove this in arbitrary characteristic, using the following desingularization theorem of Temkin. We are grateful to Benjamin Antieau for communicating to us this result.

\begin{thm}[Temkin] 
	\label{thm:temkin}
	Let $k$ be a perfect field and $V$ a perfect valuation ring over $k$. Then $V$ is the filtered colimit of its smooth $k$-subalgebras.
\end{thm}

\begin{proof} 
	We must show that every finitely generated subalgebra $A\subset V$ is contained in a smooth subalgebra of $V$. Let $K$ be the fraction field of $A$ and let $R=V\cap K$. Then $R$ is a valuation ring of $K$ centered on $A$.
	By \cite[Theorem 1.3.2]{temkin-insep}, there exists a finite purely inseparable extension $L/K$ and an algebra $A\subset A'\subset K$ such that, if $S$ is the unique valuation ring of $L$ dominating $R$, then $S$ is centered on a smooth point of the normalization $A''$ of $A'$ in $L$. 
	Since $V$ is perfect, $L$ is contained in the fraction field of $V$ and $S=V\cap L$. Hence, if $A'''\subset S$ is a localization of $A''$ that is smooth over $k$, then $A\subset A'''\subset V$, as desired.
\end{proof}

\begin{cor}\label{cor:v-rigidity}
	Let $k$ be a field and $E \in \SH(k)$. Suppose that $E$ is $\Phi$-nilpotent where $\Phi\subset\GW(k)$ is the subset of elements whose rank is invertible in $k$. Then the presheaf of spectra $E(-)\colon \Sch^{\op}_k \rightarrow \Spt$ satisfies rigidity for henselian valuation rings.
\end{cor}

\begin{proof} 
	By Lemmas \ref{lem:nilpotent} and \ref{lem:GW-nilpotence}, the exponential characteristic of $k$ acts invertibly in $E$, so that $E(-)$ is invariant under universal homeomorphisms \cite[Corollary 2.1.5]{shperf}. Let $V$ be a henselian valuation ring over $k$. Then $V^\perf$ is a henselian valuation ring over $k^\perf$, since henselian pairs are preserved by integral base change. Thus, we can assume $k$ and $V$ perfect. 
	In this case, the result follows from Theorems \ref{thm:smooth-rigidity} and \ref{thm:temkin}.
\end{proof}

Recall that the effective homotopy $t$-structure on $\SH(S)$ is defined as follows: $\SH(S)^\eff_{\geq 0}\subset\SH(S)$ is the full subcategory generated under colimits and extensions by $\Sigma^\infty_\T X_+$ for $X\in\Sm_S$. It follows immediately from the definitions that $E$ is $n$-truncated in the effective homotopy $t$-structure if and only if $E(X)\in\Spt_{\leq n}$ for every smooth $S$-scheme $X$.

\begin{cor} \label{cor:truncative} 
	Let $k$ be a field of exponential characteristic $c$ and let $E\in\SH(k)[\frac 1c]$. If $E$ is $n$-truncated in the effective homotopy $t$-structure, then the presheaf of spectra $E(-)\colon \Sch^{\op}_k \rightarrow \Spt$ factors through $\Spt_{\leq n}$.
\end{cor}

\begin{proof}
	Since $E(-)$ is a finitary rh sheaf (Lemma~\ref{lem:E(-)}), it suffices to show that $E(V)$ is $n$-truncated for every valuation ring $V$ over $k$. 
	By \cite[Corollary 2.1.5]{shperf}, $E(-)$ is invariant under universal homeomorphisms, so we can assume that $V$ is perfect. Moreover, by \cite[Lemma 1.12]{SuslinDM}, for any finitely presented smooth $k^\perf$-scheme $X$, there is a smooth $k$-scheme $Y$ and a universal homeomorphism $X\to Y$; hence, $E(X)$ is $n$-truncated.
	By Theorem~\ref{thm:temkin}, $V$ is a filtered colimit of smooth $k^\perf$-algebras, so $E(V)$ is a filtered colimit of $n$-truncated spectra and hence is $n$-truncated.
\end{proof}

\begin{thm} \label{thm:main-excision} 
	Let $k$ be a field and $E \in \SH(k)$. Suppose that $E$ is $\Phi$-nilpotent where $\Phi\subset\GW(k)$ is the subset of elements whose rank is invertible in $k$.
	\begin{enumerate}
		\item The presheaf of spectra $E(-)\colon\Sch_k^{\op} \rightarrow \Spt$ satisfies Milnor excision.
		\item If $E$ is truncated in the effective homotopy $t$-structure, then the presheaf of spectra $E(-)\colon\Sch_k^{\op} \rightarrow \Spt$ satisfies cdarc descent.
	\end{enumerate}
\end{thm}

\begin{proof}
The presheaf $E(-)$ is a finitary cdh sheaf by Lemma~\ref{lem:E(-)}, and it satisfies rigidity for henselian valuation rings by Corollary~\ref{cor:v-rigidity}. Moreover, if $E$ is $n$-truncated in the effective homotopy $t$-structure, then $E(-)$ takes values in the $\infty$-category $\Spt_{\leq n}$ by Corollary~\ref{cor:truncative}, which is compactly generated by cotruncated objects. The theorem now follows from Theorem~\ref{thm:rigidity}.
\end{proof}

\begin{cor}
	Let $k$ be a field, $E \in \SH(k)$, and $n$ an integer invertible in $k$. Then the presheaf of spectra $E(-)/n\colon \Sch_k^\op\to \Spt$ satisfies Milnor excision.
\end{cor}

\begin{proof}
	Apply Theorem~\ref{thm:main-excision}(1) to the motivic spectrum $E/n$.
\end{proof}

For $A$ an abelian group, let $HA\in\SH(k)$ be the associated motivic Eilenberg–Mac Lane spectrum, and let $HA(q)=\Sigma^{0,q}HA$.
We denote by
\[
C^*_\mot(-,A(q))\colon \Sch_k^\op\to D(\Z)
\]
the cohomology theory defined by $HA(q)$.
Note that if every element of $A$ is $\Phi$-torsion for some multiplicative subset $\Phi\subset \Z$, then $HA(q)$ is $\Phi$-nilpotent.

\begin{cor}\label{cor:motivic-cohomology}
	Let $k$ be a field and $A$ a torsion abelian group such that the exponential characteristic of $k$ acts invertibly on $A$. For any $q \in\Z$, the presheaf 
	\[
	C^*_\mot(-,A(q))\colon \Sch_k^{\op} \rightarrow \D(\Z)
	\] 
	satisfies cdarc descent. In particular, it satisfies Milnor excision.
\end{cor}

\begin{proof}
	If $q<0$, it follows from Corollary~\ref{cor:truncative} that $C^*_\mot(-,A(q))=0$, so we may assume $q\geq 0$. The motivic spectrum $HA(q)$ is then $0$-truncated in the effective homotopy $t$-structure. Indeed, if $p<0$ and $X\in\Sm_k$, then $H^{p}_\mot(X,A(q))=H^p_\et(X,A(q))=0$ by \cite[Theorem 6.17]{Voevodsky:2008}. Thus we may apply Theorem~\ref{thm:main-excision}(2) to $HA(q)$.
\end{proof}

\begin{rem}
	If $A$ is a $\Z[\tfrac 1c]$-module where $c$ is the exponential characteristic of $k$,
	then $C^*_\mot(-,A(q))\colon\Sch_k^\op\to D(\Z)$ is the cdh sheafification of the finitary extension of the presheaf $\lvert z(\A^q,0)(\Delta^\bullet\times-)\otimes A\rvert$ on $\Sch_k^\fp$. This follows from \cite[Theorem 4.2.9]{sv-cycleschow} and \cite[Theorems 5.1 and 8.4]{CDintegral}.
\end{rem}

\bibliographystyle{alphamod}

\let\mathbb=\mathbf

{\small
\bibliography{references}
}

\parskip 0pt

\end{document}

%% file: preamble.tex
\usepackage[utf8]{inputenc}
\usepackage{url}

\usepackage[expansion=false]{microtype}
\usepackage{amssymb,amsmath,amsopn,amsxtra,amsthm,amsfonts}
\usepackage{xr}
\usepackage[mathcal]{euscript}
\usepackage{mathalfa}

\usepackage[english, status=final, nomargin, inline]{fixme}
\fxusetheme{color}

\usepackage[pdfusetitle,unicode,hidelinks]{hyperref}

\usepackage{enumitem}
\setlist[enumerate,1]{label={(\arabic*)},itemsep=\parskip} %,leftmargin=0pt
\setlist[itemize,1]{itemsep=\parskip} %,leftmargin=0pt
\newlist{thmlist}{enumerate}{2}
\setlist[thmlist,1]{label={\em(\roman*)},ref={(\roman*)},%
  itemsep=\parskip,leftmargin=*,align=left}
\setlist[thmlist,2]{label={\em(\alph*)},ref={(\alph*)},%
  itemsep=\parskip,leftmargin=*,align=left,topsep=0.1cm}
\newlist{remlist}{enumerate}{2}
\setlist[remlist,1]{label={(\roman*)},ref={(\roman*)},itemsep=\parskip,%
  leftmargin=*,align=left}
\setlist[remlist,2]{label={(\alph*)},ref={(\alph*)},itemsep=\parskip,%
  leftmargin=*,align=left,topsep=0.1cm}

\makeatletter
\let\c@equation\c@subsubsection

\makeatother

\newtheorem{cor}[subsubsection]{Corollary}
\newtheorem{lem}[subsubsection]{Lemma}
\newtheorem{prop}[subsubsection]{Proposition}

\newtheorem{thm}[subsubsection]{Theorem}

\newtheorem*{claim*}{Claim}

\newtheorem*{thmA}{Theorem A}
\newtheorem*{thmB}{Theorem B}
\newtheorem*{thmC}{Theorem C}
\newtheorem*{thmD}{Theorem D}

\theoremstyle{definition}
\newtheorem{defn}[subsubsection]{Definition}

\newtheorem{rem}[subsubsection]{Remark}

\newtheorem{exam}[subsubsection]{Example}

\renewcommand{\eqref}[1]{(\ref{#1})}

\usepackage{tikz}
\usetikzlibrary{matrix}
\usepackage{tikz-cd}

\newcommand{\changelocaltocdepth}[1]{%
  \addtocontents{toc}{\protect\setcounter{tocdepth}{#1}}%
  \setcounter{tocdepth}{#1}}

\usepackage{xspace}

\usepackage{xifthen}

\usepackage{xparse}

\newcommand{\nc}{\newcommand}
\nc{\renc}{\renewcommand}
\nc{\ssec}{\subsection}
\nc{\sssec}{\subsubsection}
\nc{\on}{\operatorname}
\nc{\term}[1]{#1\xspace}

\newcommand{\m}{\mathfrak{m}}
\newcommand{\p}{\mathfrak{p}}
\newcommand{\q}{\mathfrak{q}}

\DeclareMathSymbol{A}{\mathalpha}{operators}{`A}
\DeclareMathSymbol{B}{\mathalpha}{operators}{`B}
\DeclareMathSymbol{C}{\mathalpha}{operators}{`C}
\DeclareMathSymbol{D}{\mathalpha}{operators}{`D}
\DeclareMathSymbol{E}{\mathalpha}{operators}{`E}
\DeclareMathSymbol{F}{\mathalpha}{operators}{`F}
\DeclareMathSymbol{G}{\mathalpha}{operators}{`G}
\DeclareMathSymbol{H}{\mathalpha}{operators}{`H}
\DeclareMathSymbol{I}{\mathalpha}{operators}{`I}
\DeclareMathSymbol{J}{\mathalpha}{operators}{`J}
\DeclareMathSymbol{K}{\mathalpha}{operators}{`K}
\DeclareMathSymbol{L}{\mathalpha}{operators}{`L}
\DeclareMathSymbol{M}{\mathalpha}{operators}{`M}
\DeclareMathSymbol{N}{\mathalpha}{operators}{`N}
\DeclareMathSymbol{O}{\mathalpha}{operators}{`O}
\DeclareMathSymbol{P}{\mathalpha}{operators}{`P}
\DeclareMathSymbol{Q}{\mathalpha}{operators}{`Q}
\DeclareMathSymbol{R}{\mathalpha}{operators}{`R}
\DeclareMathSymbol{S}{\mathalpha}{operators}{`S}
\DeclareMathSymbol{T}{\mathalpha}{operators}{`T}
\DeclareMathSymbol{U}{\mathalpha}{operators}{`U}
\DeclareMathSymbol{V}{\mathalpha}{operators}{`V}
\DeclareMathSymbol{W}{\mathalpha}{operators}{`W}
\DeclareMathSymbol{X}{\mathalpha}{operators}{`X}
\DeclareMathSymbol{Y}{\mathalpha}{operators}{`Y}
\DeclareMathSymbol{Z}{\mathalpha}{operators}{`Z}

\nc{\sA}{\ensuremath{\mathcal{A}}\xspace}
\nc{\sB}{\ensuremath{\mathcal{B}}\xspace}
\nc{\sC}{\ensuremath{\mathcal{C}}\xspace}
\nc{\sD}{\ensuremath{\mathcal{D}}\xspace}
\nc{\sE}{\ensuremath{\mathcal{E}}\xspace}
\nc{\sF}{\ensuremath{\mathcal{F}}\xspace}
\nc{\sG}{\ensuremath{\mathcal{G}}\xspace}
\nc{\sH}{\ensuremath{\mathcal{H}}\xspace}
\nc{\sI}{\ensuremath{\mathcal{I}}\xspace}
\nc{\sJ}{\ensuremath{\mathcal{J}}\xspace}
\nc{\sK}{\ensuremath{\mathcal{K}}\xspace}
\nc{\sL}{\ensuremath{\mathcal{L}}\xspace}
\nc{\sM}{\ensuremath{\mathcal{M}}\xspace}
\nc{\sN}{\ensuremath{\mathcal{N}}\xspace}
\nc{\sO}{\ensuremath{\mathcal{O}}\xspace}
\nc{\sP}{\ensuremath{\mathcal{P}}\xspace}
\nc{\sQ}{\ensuremath{\mathcal{Q}}\xspace}
\nc{\sR}{\ensuremath{\mathcal{R}}\xspace}
\nc{\sS}{\ensuremath{\mathcal{S}}\xspace}
\nc{\sT}{\ensuremath{\mathcal{T}}\xspace}
\nc{\sU}{\ensuremath{\mathcal{U}}\xspace}
\nc{\sV}{\ensuremath{\mathcal{V}}\xspace}
\nc{\sW}{\ensuremath{\mathcal{W}}\xspace}
\nc{\sX}{\ensuremath{\mathcal{X}}\xspace}
\nc{\sY}{\ensuremath{\mathcal{Y}}\xspace}
\nc{\sZ}{\ensuremath{\mathcal{Z}}\xspace}

\nc{\bA}{\ensuremath{\mathbf{A}}\xspace}
\nc{\bB}{\ensuremath{\mathbf{B}}\xspace}
\nc{\bC}{\ensuremath{\mathbf{C}}\xspace}
\nc{\bD}{\ensuremath{\mathbf{D}}\xspace}
\nc{\bE}{\ensuremath{\mathbf{E}}\xspace}
\nc{\bF}{\ensuremath{\mathbf{F}}\xspace}
\nc{\bG}{\ensuremath{\mathbf{G}}\xspace}
\nc{\bH}{\ensuremath{\mathbf{H}}\xspace}
\nc{\bI}{\ensuremath{\mathbf{I}}\xspace}
\nc{\bJ}{\ensuremath{\mathbf{J}}\xspace}
\nc{\bK}{\ensuremath{\mathbf{K}}\xspace}
\nc{\bL}{\ensuremath{\mathbf{L}}\xspace}
\nc{\bM}{\ensuremath{\mathbf{M}}\xspace}
\nc{\bN}{\ensuremath{\mathbf{N}}\xspace}
\nc{\bO}{\ensuremath{\mathbf{O}}\xspace}
\nc{\bP}{\ensuremath{\mathbf{P}}\xspace}
\nc{\bQ}{\ensuremath{\mathbf{Q}}\xspace}
\nc{\bR}{\ensuremath{\mathbf{R}}\xspace}
\nc{\bS}{\ensuremath{\mathbf{S}}\xspace}
\nc{\bT}{\ensuremath{\mathbf{T}}\xspace}
\nc{\bU}{\ensuremath{\mathbf{U}}\xspace}
\nc{\bV}{\ensuremath{\mathbf{V}}\xspace}
\nc{\bW}{\ensuremath{\mathbf{W}}\xspace}
\nc{\bX}{\ensuremath{\mathbf{X}}\xspace}
\nc{\bY}{\ensuremath{\mathbf{Y}}\xspace}
\nc{\bZ}{\ensuremath{\mathbf{Z}}\xspace}

\nc{\dA}{\ensuremath{\mathds{A}}\xspace}
\nc{\dB}{\ensuremath{\mathds{B}}\xspace}
\nc{\dC}{\ensuremath{\mathds{C}}\xspace}
\nc{\dD}{\ensuremath{\mathds{D}}\xspace}
\nc{\dE}{\ensuremath{\mathds{E}}\xspace}
\nc{\dF}{\ensuremath{\mathds{F}}\xspace}
\nc{\dG}{\ensuremath{\mathds{G}}\xspace}
\nc{\dH}{\ensuremath{\mathds{H}}\xspace}
\nc{\dI}{\ensuremath{\mathds{I}}\xspace}
\nc{\dJ}{\ensuremath{\mathds{J}}\xspace}
\nc{\dK}{\ensuremath{\mathds{K}}\xspace}
\nc{\dL}{\ensuremath{\mathds{L}}\xspace}
\nc{\dM}{\ensuremath{\mathds{M}}\xspace}
\nc{\dN}{\ensuremath{\mathds{N}}\xspace}
\nc{\dO}{\ensuremath{\mathds{O}}\xspace}
\nc{\dP}{\ensuremath{\mathds{P}}\xspace}
\nc{\dQ}{\ensuremath{\mathds{Q}}\xspace}
\nc{\dR}{\ensuremath{\mathds{R}}\xspace}
\nc{\dS}{\ensuremath{\mathds{S}}\xspace}
\nc{\dT}{\ensuremath{\mathds{T}}\xspace}
\nc{\dU}{\ensuremath{\mathds{U}}\xspace}
\nc{\dV}{\ensuremath{\mathds{V}}\xspace}
\nc{\dW}{\ensuremath{\mathds{W}}\xspace}
\nc{\dX}{\ensuremath{\mathds{X}}\xspace}
\nc{\dY}{\ensuremath{\mathds{Y}}\xspace}
\nc{\dZ}{\ensuremath{\mathds{Z}}\xspace}

\nc{\bbA}{\ensuremath{\mathbb{A}}\xspace}
\nc{\bbB}{\ensuremath{\mathbb{B}}\xspace}
\nc{\bbC}{\ensuremath{\mathbb{C}}\xspace}
\nc{\bbD}{\ensuremath{\mathbb{D}}\xspace}
\nc{\bbE}{\ensuremath{\mathbb{E}}\xspace}
\nc{\bbF}{\ensuremath{\mathbb{F}}\xspace}
\nc{\bbG}{\ensuremath{\mathbb{G}}\xspace}
\nc{\bbH}{\ensuremath{\mathbb{H}}\xspace}
\nc{\bbI}{\ensuremath{\mathbb{I}}\xspace}
\nc{\bbJ}{\ensuremath{\mathbb{J}}\xspace}
\nc{\bbK}{\ensuremath{\mathbb{K}}\xspace}
\nc{\bbL}{\ensuremath{\mathbb{L}}\xspace}
\nc{\bbM}{\ensuremath{\mathbb{M}}\xspace}
\nc{\bbN}{\ensuremath{\mathbb{N}}\xspace}
\nc{\bbO}{\ensuremath{\mathbb{O}}\xspace}
\nc{\bbP}{\ensuremath{\mathbb{P}}\xspace}
\nc{\bbQ}{\ensuremath{\mathbb{Q}}\xspace}
\nc{\bbR}{\ensuremath{\mathbb{R}}\xspace}
\nc{\bbS}{\ensuremath{\mathbb{S}}\xspace}
\nc{\bbT}{\ensuremath{\mathbb{T}}\xspace}
\nc{\bbU}{\ensuremath{\mathbb{U}}\xspace}
\nc{\bbV}{\ensuremath{\mathbb{V}}\xspace}
\nc{\bbW}{\ensuremath{\mathbb{W}}\xspace}
\nc{\bbX}{\ensuremath{\mathbb{X}}\xspace}
\nc{\bbY}{\ensuremath{\mathbb{Y}}\xspace}
\nc{\bbZ}{\ensuremath{\mathbb{Z}}\xspace}

\nc{\mrm}[1]{\ensuremath{\mathrm{#1}}\xspace}
\nc{\mbf}[1]{\ensuremath{\mathbf{#1}}\xspace}
\nc{\mcal}[1]{\ensuremath{\mathcal{#1}}\xspace}
\nc{\msc}[1]{\ensuremath{\mathscr{#1}}\xspace}

\renc{\bar}[1]{\overline{#1}}

\let\sectsign\S
\let\S\relax

\nc{\sub}{\subset}
\nc{\too}{\longrightarrow}
\nc{\hook}{\hookrightarrow}
\nc*{\hooklongrightarrow}{\ensuremath{\lhook\joinrel\relbar\joinrel\rightarrow}}
\nc{\hooklong}{\hooklongrightarrow}
\nc{\twoheadlongrightarrow}{\relbar\joinrel\twoheadrightarrow}
\nc{\shiso}{\approx}
\nc{\isoto}{\xrightarrow{\sim}}
\nc{\isofrom}{\xleftarrow{\sim}}
\renc{\ge}{\geqslant}
\renc{\le}{\leqslant}
\renc{\geq}{\geqslant}
\renc{\leq}{\leqslant}

\nc{\id}{\mathrm{id}}

\DeclareMathOperator{\rk}{\mathrm{rk}}
\DeclareMathOperator{\Hom}{\mathrm{Hom}}
\nc{\uHom}{\underline{\smash{\Hom}}}
\DeclareMathOperator{\Maps}{\mathrm{Maps}}

\DeclareMathOperator{\End}{\mathrm{End}}

\nc{\Pre}{\mathrm{PSh}{}}
\nc{\Shv}{\mathrm{Shv}{}}
\nc{\uEnd}{\underline{\smash{\End}}}

\renc{\lim}{\operatorname*{lim}}
\nc{\colim}{\operatorname*{colim}}
\nc{\Cofib}{\on{Cofib}}
\nc{\Fib}{\on{Fib}}
\nc{\initial}{\varnothing}
\nc{\op}{\mathrm{op}}

\let\bigcoprod=\coprod
\renc{\coprod}{\sqcup}

%% file: macros.tex
\nc{\bDelta}{\mbf{\Delta}}
\nc{\DM}{\mbf{DM}}
\nc{\eff}{\mathrm{eff}}
\nc{\veff}{\mathrm{veff}}
\nc{\cyc}{{\mrm{cyc}}}
\nc{\corr}{{\on{corr}}}
\nc{\ft}{\mrm{ft}}
\nc{\flf}{\mrm{flf}}
\nc{\fet}{{\mrm{f\acute et}}}
\nc{\fsyn}{{\mrm{fsyn}}}
\nc{\syn}{{\mrm{syn}}}
\nc{\lci}{{\mrm{lci}}}
\nc{\Perf}{\mbf{Perf}}
\nc{\perf}{\mrm{perf}}
\nc{\oblv}{\mrm{oblv}}
\nc{\exact}{\on{exact}}

\nc{\F}{{\on{F}}}
\nc{\clopen}{{\mrm{clopen}}}
\nc{\B}{\mrm{B}}
\nc{\D}{\mrm{D}}
\nc{\Fin}{\on{Fin}}
\nc{\fin}{\mrm{fin}}
\nc{\Cut}{\on{Cut}}
\nc{\Cart}{\on{Cart}}
\nc{\pairs}{\mathsf{pairs}}
\nc{\Pairs}{\mathrm{Pair}}
\nc{\Trip}{\mathrm{Trip}}
\nc{\Lab}{\mathrm{Lab}}
\nc{\SL}{\mathrm{SL}}
\nc{\coCart}{\mathrm{coCart}}
\nc{\RKE}{\mathrm{RKE}}
\nc{\strict}{\mathrm{strict}}
\nc{\Emb}{\mathrm{Emb}}
\nc{\Split}{\mathrm{Split}}
\nc{\Set}{\mathrm{Set}}
\nc{\sSets}{\mathrm{sSets}}
\nc{\pb}{\mathrm{pb}}
\nc{\fib}{\mathrm{fib}}
\nc{\diff}{\mrm{diff}}
\nc{\gp}{\mrm{gp}}
\nc{\map}{\mrm{map}}

\nc{\mgp}{\mrm{mot-gp}}
\nc{\FSyn}{\mrm{FSyn}}
\nc{\FEt}{\mrm{FEt}}
\nc{\Spc}{\mrm{Spc}}
\nc{\Ob}{\mrm{Ob}}
\nc{\Spt}{\mrm{Spt}}
\nc{\T}{\bT}
\nc{\suspinf}{\Sigma^\infty}
\nc{\h}{\mrm{h}}
\nc{\uhom}{\underline{\mathrm{Hom}}}
\nc{\umap}{\underline{\mathrm{Maps}}}
\renc{\H}{\bH}
\nc{\Einfty}{{\sE_\infty}}
\nc{\Eone}{{\sE_1}}
\nc{\Stab}{\mrm{Stab}}
\nc{\lax}{{\mrm{lax}}}
\nc{\cocart}{{\mrm{cocart}}}
\nc{\Sch}{\mrm{Sch}}
\nc{\Fr}{\on{Fr}}
\nc{\A}{\mathbf{A}}
\nc{\N}{\mathbf{N}}
\nc{\Z}{\mathbf{Z}}
\nc{\Q}{\mathbf{Q}}
\nc{\Oo}{\mathcal{O}} 
\nc{\Fscr}{\mathcal{F}}
\nc{\Gscr}{\mathcal{G}}
\nc{\Ll}{\mathcal{L}} 
\nc{\Mm}{\mathcal{M}} 
\nc{\mm}{\mathrm{m}} 
\nc{\K}{\mrm{K}} 
\nc{\W}{\mrm{W}} 
\nc{\red}{{\on{red}}}
\nc{\Voev}{{\on{Voev}}}
\nc{\Corr}{\mrm{Corr}}
\nc{\Span}{\mathbf{Corr}}
\nc{\Gap}{\mrm{Gap}}
\nc{\Corrfr}{\Corr^{\fr}}
\nc{\Corrvfr}{\Corr^{\Vfr}}
\nc{\Spec}{\on{Spec}}
\nc{\Sm}{\on{Sm}}
\nc{\Gm}{\mathbf{G}_{\on{m}}}
\renc{\P}{\bP}
\nc{\nis}{\mathrm{nis}}
\nc{\zar}{\mathrm{zar}}
\nc{\et}{\mathrm{\acute et}}
\nc{\all}{\mathrm{all}}
\nc{\fold}{\mathrm{fold}}
\nc{\Fun}{\mathrm{Fun}}
\nc{\Ho}{\mathrm{Ho}}
\nc{\Segal}{\mathrm{Segal}}
\nc{\Mon}{\mrm{Mon}{}}
\nc{\Ab}{\mrm{Ab}}
\nc{\Sh}{\on{Sh}}
\nc{\M}{\mrm{M}}
\nc{\Lhtp}{L_{\A^1}}
\nc{\Lmot}{L_{\mrm{mot}}}
\nc{\mot}{\mrm{mot}}
\nc{\SH}{\mbf{SH}}
\nc{\RR}{\mbf{R}}
\nc{\CC}{\mbf{C}}
\nc{\Mod}{\mbf{Mod}}
\nc{\QCoh}{\mbf{QCoh}}
\nc{\MonUnit}{\mbf{1}}
\nc{\1}{\mbf{1}}
\nc{\tr}{\on{tr}}
\nc{\cotr}{\mrm{cotr}}
\nc{\vop}{\mrm{vop}}
\nc{\fr}{{\on{fr}}}
\nc{\Ar}{\mrm{Ar}}
\nc{\Vfr}{\on{Vfr}}
\nc{\frdiff}{{\on{frdiff}}}
\nc{\frGys}{\on{frGys}}
\nc{\SHfr}{\SH^{\fr}}
\nc{\SHfrdiff}{\SH^{\frdiff}}
\nc{\SHfrGys}{\SH^{\frGys}}
\nc{\InftyCat}{(\mathrm{\infty,1)\textnormal{-}Cat}}
\nc{\TriCat}{\mathrm{TriCat}}
\nc{\oneCat}{\mathrm{1\textnormal{-}Cat}}
\nc{\Cat}{\mathrm{Cat}}
\nc{\Th}{\on{Th}}

\nc{\CMon}{\mrm{CMon}{}}
\nc{\CAlg}{\mrm{CAlg}{}}
\nc{\MGL}{\mrm{MGL}}
\nc{\Seg}{\mrm{Seg}{}}
\nc{\GW}{\mrm{GW}{}}
\nc{\Tw}{\mrm{Tw}}
\nc{\sslash}{/\mkern-6mu/}
\nc{\PrL}{\mrm{Pr}^\mrm{L}}
\nc{\PrR}{\mrm{Pr}^\mrm{R}}
\nc{\pr}{\mrm{pr}}
\let\phi\varphi

\nc\efr{\mrm{efr}}
\nc\nfr{\mrm{nfr}}
\nc\dfr{\mrm{fr}}
\nc\tfr{\mrm{tfr}}
\nc\Vect{\mrm{Vect}}
\nc\sVect{\mrm{sVect}}
\nc{\fix}{\mrm{fix}}
\nc{\ho}{\mrm{h}}
\nc\Mfd{\mrm{Mfd}}
\nc{\PSh}{\mrm{PSh}}
\nc{\hzmw}{H \tilde\Z{}}
\nc{\Cor}{\mrm{Cor}{}}
\nc{\cormw}{\mrm{\widetilde{Cor}}{}}
\nc{\Chw}{\mrm{\widetilde{CH}}{}}
\nc{\Ex}{\mrm{Ex}}
\nc{\BM}{\mrm{BM}}

\nc{\Pic}{\mrm{Pic}}
\nc{\pur}{\mathfrak p}
\nc{\angles}[1]{\langle #1\rangle}
\nc{\inv}[1]{[\tfrac{1}{#1}]}
\nc{\pinv}{\inv{p}}
\nc{\cinv}{\inv{p}}
\nc{\Sph}{\on{Sph}}
\nc{\KGL}{\mrm{KGL}}
\nc{\KH}{\mrm{KH}}
\nc{\Flag}{\mrm{Flag}}
\nc{\Pro}{\mrm{Pro}}
\nc{\Frac}{\mrm{Frac}}
\newcommand{\fp}{\mathrm{fp}}

\nc{\arc}{\mrm{arc}}
\nc{\rarc}{\mrm{rarc}}
\nc{\cdarc}{\mrm{cdarc}}
\nc{\vv}{\mrm{v}}
\nc{\rv}{\mrm{rv}}
\nc{\cdv}{\mrm{cdv}}
\nc{\hh}{\mrm{h}}
\nc{\cdh}{\mrm{cdh}}
\nc{\rh}{\mathrm{rh}}
\nc{\Et}{\mathrm{Et}}
\nc{\Nis}{\mathrm{Nis}}
\nc{\Zar}{\mathrm{Zar}}
\nc{\cdp}{\mathrm{cdp}}

\nc{\RZ}{\mathrm{RZ}}
\nc{\qcqs}{\mathrm{qcqs}}
\nc{\aff}{\mathrm{aff}}
\nc{\cl}{\mathrm{cl}}
\nc{\Val}{\mathrm{Val}}
\nc{\GFin}{\mathrm{GFin}{}}
\nc{\Proj}{\mathrm{Proj}}

\nc{\inftyCat}{\term{$\infty$-category}}
\nc{\inftyCats}{\term{$\infty$-categories}}

\nc{\inftyOneCat}{\term{$(\infty,1)$-category}}
\nc{\inftyOneCats}{\term{$(\infty,1)$-categories}}

\nc{\inftyGrpd}{\term{$\infty$-groupoid}}
\nc{\inftyGrpds}{\term{$\infty$-groupoids}}

\nc{\inftyTop}{\term{$\infty$-topos}}
\nc{\inftyTops}{\term{$\infty$-toposes}}

\nc{\inftyTwoCat}{\term{$(\infty,2)$-category}}
\nc{\inftyTwoCats}{\term{$(\infty,2)$-categories}}

%% file: excisev3.bbl
\providecommand{\bysame}{\leavevmode\hbox to3em{\hrulefill}\thinspace}
\providecommand{\MR}{\relax\ifhmode\unskip\space\fi MR }
\providecommand{\MRhref}[2]{%
  \href{http://www.ams.org/mathscinet-getitem?mr=#1}{#2}
}
\providecommand{\href}[2]{#2}
\begin{thebibliography}{EHIK20}
\providecommand{\url}[1]{\href{#1}{{\def~{\textasciitilde}\tt #1}}}

\bibitem[AD18]{sh-rigid}
A.~Ananyevskiy and A.~Druzhinin, \emph{Rigidity for linear framed presheaves
  and generalized motivic cohomology theories}, Adv. Math. \textbf{333} (2018),
  pp.~423--462,  \url{https://doi.org/10.1016/j.aim.2018.05.013}

\bibitem[AHW16]{AHW}
A.~Asok, M.~Hoyois, and M.~Wendt., \emph{Affine representability results in
  {${\mathbb A}^1$}-homotopy theory {I}: vector bundles}, to appear in Duke
  Math. J., 2016, \href{http://arxiv.org/abs/1506.07093}{{\sf
  arXiv:1506.07093}}

\bibitem[BH18]{norms}
T.~Bachmann and M.~Hoyois, \emph{Norms in motivic homotopy theory}, To appear
  in Asterisque, 2018, \href{http://arxiv.org/abs/1711.03061}{{\sf
  arXiv:1711.03061}}

\bibitem[BM20]{arc}
B.~Bhatt and A.~Mathew, \emph{The arc topology}, 2020,
  \href{http://arxiv.org/abs/arXiv:1807.04725v3}{{\sf arXiv:1807.04725v3}}

\bibitem[Bou75]{BourbakiCAlg}
N.~Bourbaki, \emph{\'{E}l\'ements de math\'ematique. {A}lg\`ebre commutative.
  {C}hapitres 5 \`a 7}, Springer, Berlin, 1975

\bibitem[BS17]{bhatt-scholze-proj}
B.~Bhatt and P.~Scholze, \emph{Projectivity of the {W}itt vector affine
  {G}rassmannian}, Invent. Math. \textbf{209} (2017), no.~2, pp.~329--423,
  \url{https://doi.org/10.1007/s00222-016-0710-4}

\bibitem[CD15]{CDintegral}
D.-C. Cisinski and F.~D{\'e}glise, \emph{Integral mixed motives in equal
  characteristic}, Doc. Math., Extra Volume: Alexander S. Merkurjev's Sixtieth
  Birthday (2015), pp.~145--194

\bibitem[CM19]{clausen-mathew}
D.~Clausen and A.~Mathew, \emph{Hyperdescent and {\'e}tale K-theory}, 2019,
  \href{http://arxiv.org/abs/arXiv:1905.06611}{{\sf arXiv:1905.06611}}

\bibitem[EHIK20]{sh-excision}
E.~Elmanto, M.~Hoyois, R.~Iwasa, and S.~Kelly, \emph{Milnor excision for
  motivic spectra}, in preparation, 2020

\bibitem[EK19]{shperf}
E.~Elmanto and A.~A. Khan, \emph{Perfection in motivic homotopy theory}, Proc.
  London. Math. Soc. \textbf{120} (2019), pp.~28--38, preprint
  \href{http://arxiv.org/abs/1812.07506}{{\sf arXiv:1812.07506}}

\bibitem[Fer03]{ferrand}
D.~Ferrand, \emph{Conducteur, descente et pincement}, Bull. Soc. Math. France
  \textbf{131} (2003), no.~4, pp.~553--585

\bibitem[GD71]{EGA1new}
A.~Grothendieck and J.~A. Dieudonn{\'e}, \emph{{\'E}l{\'e}ments de
  G{\'e}om{\'e}trie Alg{\'e}brique I}, Springer-Verlag, 1971

\bibitem[GK15]{GabberKelly}
O.~Gabber and S.~Kelly, \emph{Points in algebraic geometry}, J. Pure Appl.
  Algebra \textbf{219} (2015), no.~10, pp.~4667--4680

\bibitem[GL01]{goodwillie-lichtenbaum}
T.~G. Goodwillie and S.~Lichtenbaum, \emph{A cohomological bound for the
  {$h$}-topology}, Amer. J. Math. \textbf{123} (2001), no.~3, pp.~425--443,
  \url{http://muse.jhu.edu/journals/american_journal_of_mathematics/v123/123.3goodwillie.pdf}

\bibitem[GR71]{GrusonRaynaud}
L.~Gruson and M.~Raynaud, \emph{Crit{\`e}res de platitude et de
  projectivit{\'e}}, Invent. Math. \textbf{13} (1971), no.~1-2, pp.~1--89

\bibitem[HJ65]{HJ}
M.~Henriksen and M.~Jerison, \emph{The space of minimal prime ideals of a
  commutative ring}, Trans. Amer. Math. Soc. \textbf{115} (1965), pp.~110--130

\bibitem[HK18]{huber-kelly}
A.~Huber and S.~Kelly, \emph{Differential forms in positive characteristic,
  {II}: cdh-descent via functorial {R}iemann-{Z}ariski spaces}, Algebra Number
  Theory \textbf{12} (2018), no.~3, pp.~649--692,
  \url{https://doi.org/10.2140/ant.2018.12.649}

\bibitem[Hoy14]{HoyoisGLV}
M.~Hoyois, \emph{A quadratic refinement of the
  {G}rothendieck--{L}efschetz--{V}erdier trace formula}, Algebr. Geom. Topol.
  \textbf{14} (2014), no.~6, pp.~3603--3658

\bibitem[Hoy17]{hoyois-sixops}
M.~Hoyois, \emph{The six operations in equivariant motivic homotopy theory},
  Adv. Math. \textbf{305} (2017), pp.~197--279

\bibitem[Jaf60]{Jaffard}
P.~Jaffard, \emph{Th{\'e}orie de la dimension dans les anneaux de
  polyn{\^o}mes}, M{\'e}morial des sciences math{\'e}matiques \textbf{146}
  (1960)

\bibitem[Kel19]{shane-better}
S.~Kelly, \emph{A better comparison of cdh- and ldh-cohomologies}, Nagoya Math.
  J. \textbf{236} (2019), pp.~183--213, preprint
  \href{http://arxiv.org/abs/1807.00158}{{\sf arXiv:1807.00158}}

\bibitem[KM18]{kelly-morrow}
S.~Kelly and M.~Morrow, \emph{$K$-theory of valuation rings}, 2018,
  \href{http://arxiv.org/abs/1810.12203v1}{{\sf arXiv:1810.12203v1}}

\bibitem[LT19]{land-tamme}
M.~Land and G.~Tamme, \emph{On the $K$-theory of pullbacks}, Ann. Math.
  \textbf{190} (2019), no.~3, pp.~877--930, preprint
  \href{http://arxiv.org/abs/1808.05559}{{\sf arXiv:1808.05559}}

\bibitem[Lur03]{LurieTopoi}
J.~Lurie, \emph{On Infinity Topoi}, 2003,
  \href{http://arxiv.org/abs/math/0306109v2}{{\sf arXiv:math/0306109v2}}

\bibitem[Lur17]{HTT}
\bysame, \emph{Higher Topos Theory}, April 2017,
  \url{http://www.math.harvard.edu/~lurie/papers/HTT.pdf}

\bibitem[Lur18]{SAG}
\bysame, \emph{Spectral Algebraic Geometry}, February 2018,
  \url{http://www.math.harvard.edu/~lurie/papers/SAG-rootfile.pdf}

\bibitem[MH73]{milnor1973symmetric}
J.~W. Milnor and D.~Husemoller, \emph{Symmetric bilinear forms}, vol.~60,
  Springer, 1973

\bibitem[Mil71]{milnor-ktheory}
J.~Milnor, \emph{Introduction to algebraic {$K$}-theory}, pp.~xiii+184, Annals
  of Mathematics Studies, No. 72

\bibitem[Mor18]{Morrow}
M.~Morrow, \emph{Pro unitality and pro excision in algebraic K-theory and
  cyclic homology}, J. reine angew. Math \textbf{736} (2018), pp.~95--139

\bibitem[Nag53]{nagata-hensel}
M.~Nagata, \emph{On the theory of {H}enselian rings}, Nagoya Math. J.
  \textbf{5} (1953), pp.~45--57,
  \url{http://projecteuclid.org/euclid.nmj/1118799392}

\bibitem[Ryd10]{Rydh-submerse}
D.~Rydh, \emph{Submersions and effective descent of {\'e}tale morphisms}, Bull.
  Soc. Math. France \textbf{138} (2010), no.~2, pp.~181--230, preprint
  \href{http://arxiv.org/abs/0710.2488}{{\sf arXiv:0710.2488}}

\bibitem[Stacks]{stacks}
{The Stacks Project Authors}, \emph{The Stacks Project}, 2019,
  \url{http://stacks.math.columbia.edu}

\bibitem[Sus17]{SuslinDM}
A.~Suslin, \emph{Motivic complexes over nonperfect fields}, Ann. K-Theory
  \textbf{2} (2017), no.~2, pp.~277--302

\bibitem[SV00]{sv-cycleschow}
A.~Suslin and V.~Voevodsky, \emph{Relative cycles and {C}how sheaves}, Cycles,
  transfers, and motivic homology theories, Ann. of Math. Stud., vol. 143,
  Princeton Univ. Press, 2000, pp.~10--86

\bibitem[Tem10]{temkin-curves}
M.~Temkin, \emph{Stable modification of relative curves}, J. Algebraic Geom.
  \textbf{19} (2010), pp.~603--677

\bibitem[Tem11]{temkin-rz}
\bysame, \emph{Relative {R}iemann--{Z}ariski spaces}, Israel J. Math.
  \textbf{185} (2011), pp.~1--42,
  \url{https://doi.org/10.1007/s11856-011-0099-0}

\bibitem[Tem13]{temkin-insep}
\bysame, \emph{Inseparable local uniformization}, J. Algebra \textbf{373}
  (2013), pp.~65--119,  \url{https://doi.org/10.1016/j.jalgebra.2012.09.023}

\bibitem[Voe96]{HS1}
V.~Voevodsky, \emph{Homology of Schemes}, Selecta Math. (N.S.) \textbf{2}
  (1996), no.~1, pp.~111--153, preprint
  \href{http://www.math.uiuc.edu/K-theory/0031}{{\sf K-theory:0031}}

\bibitem[Voe10a]{cd-voe}
\bysame, \emph{Homotopy theory of simplicial sheaves in completely decomposable
  topologies}, J. Pure Appl. Algebra \textbf{214} (2010), no.~8,
  pp.~1384--1398,  \url{https://doi.org/10.1016/j.jpaa.2009.11.004}

\bibitem[Voe10b]{unstable-voe}
\bysame, \emph{Unstable motivic homotopy categories in {N}isnevich and
  cdh-topologies}, J. Pure Appl. Algebra \textbf{214} (2010), no.~8,
  pp.~1399--1406,  \url{https://doi.org/10.1016/j.jpaa.2009.11.005}

\bibitem[Voe11]{Voevodsky:2008}
\bysame, \emph{On motivic cohomology with $\mathbb{Z}/l$-coefficients}, Ann.
  Math. \textbf{174} (2011), no.~1, pp.~401--438, preprint
  \href{http://arxiv.org/abs/0805.4430}{{\sf arXiv:0805.4430}}

\bibitem[Wei89]{Weibel}
C.~A. Weibel, \emph{Homotopy algebraic $K$-theory}, Algebraic K-Theory and
  Number Theory, Contemp. Math., vol.~83, AMS, 1989, pp.~461--488

\bibitem[Zar40]{Zariski}
O.~Zariski, \emph{Local uniformization on algebraic varieties}, Ann. Math.
  \textbf{41} (1940), no.~4, pp.~852--896

\end{thebibliography}
